\documentclass[11pt,reqno]{amsart}
\usepackage{fullpage}
\usepackage{times}
\usepackage{amsmath,amssymb,amsthm,url}
\usepackage[utf8]{inputenc}
\usepackage[english]{babel}
\usepackage{comment}
\usepackage{enumerate}

\usepackage{graphicx}
\usepackage{mathrsfs}
\usepackage[colorlinks=true, pdfstartview=FitH, linkcolor=blue, citecolor=blue, urlcolor=blue]{hyperref}
\usepackage{mathtools}

\newtheorem{theorem}{Theorem}[section]
\newtheorem{lemma}[theorem]{Lemma}
\newtheorem{proposition}[theorem]{Proposition}
\newtheorem{corollary}[theorem]{Corollary}

\newtheorem{claim}[theorem]{Claim}
\newenvironment{poc}{\begin{proof}[Proof of claim]}{\end{proof}}
\theoremstyle{definition}

\newtheorem{remark}[theorem]{Remark}
\newcommand{\Ent}{\mathrm H}
\newcommand{\Q}{\mathbb{Q}}
\newcommand{\R}{\mathbb{R}} 
\newcommand{\N}{\mathbb{N}}
\newcommand{\Z}{\mathbb{Z}}
\newcommand{\cO}{\mathcal{O}}

\newcommand{\eps}{\varepsilon}
\newcommand{\F}{\mathbb F}

\newcommand{\1}{\mathbf 1}

\title{A combinatorial large sieve for Sidon sets, distances, and norm forms}
\author{Ernie Croot}
\address{School of Mathematics\\ Georgia Institute of Technology\\ GA\\ United States}
\email{ernest.croot@math.gatech.edu}
\author{Junzhe Mao}
\address{School of Mathematics\\ Georgia Institute of Technology\\ GA\\ United States}
\email{jmao87@gatech.edu}
\author{Cosmin Pohoata}
\address{Department of Mathematics, Emory University, GA, United States}
\email{cosmin.pohoata@emory.edu}
\author{Adam Sheffer}
\address{Department of Mathematics, Baruch College, City University of New York, NY, United States}
\email{adamsh@gmail.com}
\author{Chi Hoi Yip}
\address{School of Mathematics\\ Georgia Institute of Technology\\ GA\\ United States \& Department of Mathematics, Hong Kong University of Science and Technology, Clear Water Bay, Hong Kong}
\email{cyip30@gatech.edu}
\begin{document}

\begin{abstract}
We develop a new combinatorial large sieve method for sets with bounded algebraic multiplicities.  The method exploits algebraic splitting modulo
many small primes: local congruence branching produces many modular
collisions, while global bounded-multiplicity hypotheses force these collisions
to be rare.

As a first application, we prove that every Sidon subset
$A\subset\{1^2,\ldots,N^2\}$ satisfies
\[
        |A|
        \le
        N\exp\left(
        -c\frac{\log N}{\log\log N}
        \right)
\]
for some absolute constant $c>0$.  This gives the first
super-polylogarithmic saving for a classical problem of Alon and Erd\H{o}s.

As a second application, we establish new upper bounds for two grid-distance
problems.  We show that the largest subset of $[N]^2$ with no repeated distance has size at most $N\exp\left(-c\log N/\log\log N\right)$, giving the first progress in over thirty years on a problem of Erd\H{o}s and Guy.  The same method also gives a similar saving for subsets of $[N]^2$ with no isosceles triangles, a problem recently popularized by Ellenberg and by the PatternBoost work of Charton, Ellenberg, Wagner, and Williamson.

We then develop an entropic version of the method. This gives bounds
for $B_2[g]$-sets in the squares and for analogous bounded-multiplicity problems associated with norm forms over arbitrary number fields. More importantly, this new method also allows us to establish the first nontrivial bounds for $B_3[g]$-sets in the cubes and $B_4[g]$-sets in the fourth powers.
\end{abstract}

\maketitle

\section{Introduction} 
Sieve methods are powerful, versatile analytic methods designed to quantify the distribution of an arithmetic set after it has been sifted by a family of congruence classes modulo various primes. They apply to the case when a set is constrained to omit some residue classes modulo many primes, and they have profound applications in analytic number theory, arithmetic combinatorics, and Diophantine equations. 

The large sieve and larger sieve are among the basic tools for
turning local restrictions into global upper bounds.  In their classical
form, they say that if a set $A$ of integers omits \emph{many} residue classes modulo
many primes, then $|A|$ must be small. This philosophy is believed to have a strong inverse aspect: either $|A|$ is {\it much smaller} than the upper bound given by various sieve methods such as the large/larger sieve, or else $A$ has some {\it algebraic} structure. This is known as the inverse sieve conjecture, appearing for example in work of Helfgott--Venkatesh \cite{HelfgottVenkatesh}, Walsh \cite{WalshInvSieve, WalshAlgebraicity}, Green--Harper \cite{GreenHarper}, Shao \cite{S15, S16}, and Croot--Mao--Yip \cite{CMY}.

The present paper develops a complementary method for situations where the
set under study also has a bounded algebraic-multiplicity property.  The guiding
principle is simple.  Suppose that an algebraic congruence
\[
        F(\mathbf x)\equiv0\pmod p
\]
splits into many linear pieces modulo many primes \(p\).  Passing to products
of such primes produces many compatible local branches.  If a large set is
well distributed modulo these products, then Cauchy--Schwarz, or a direct
counting argument, forces many pairs or tuples whose algebraic value is
divisible by the modulus.  On the other hand, a bounded-multiplicity
hypothesis controls how many such pairs or tuples can exist.  If the set is not
well distributed modulo any of the relevant products, then an entropy deficit,
via Shearer's inequality, forces the set to be small. We
also develop a weighted entropic method that allows us to bypass this ``spread-versus-entropy'' dichotomy and establish stronger bounds in some cases.
These new methods allow us to improve the bound coming from direct applications of the large/larger sieve in some cases. Moreover, sometimes these direct applications only yield trivial bounds, while our methods provide nontrivial bounds.

More precisely, the method has two related forms.  The first is a direct
``subspace large-sieve'', which is strongest when the forbidden structure gives
pointwise control of differences, as in the Sidon problem for squares and the
no-repeated-distance problem in grids.  The second is an entropy-enhanced
large sieve, which applies when one only has bounded sum multiplicity, as in
\(B_h[g]\)-type problems and norm-form representation problems. We refer to Remark~\ref{rem:weightsubspace} for comparison of the two forms of the method.

In the following four subsections, we introduce background and state our new results.

\subsection{Ill-distributed Sidon sets}
A set \(A\) in an abelian group is called a \emph{Sidon set}, or a
\emph{\(B_2\)-set}, if all sums \(a+a'\), with \(a,a'\in A\), are distinct up
to the trivial symmetry \(a+a'=a'+a\). Equivalently, in a torsion-free
abelian group, every nonzero difference has at most one representation
as \(a-a'\) with \(a,a'\in A\).

The classical problem asks for the largest size of a Sidon subset of
\([N]=\{1,\ldots,N\}\). If this maximum is denoted by \(s(N)\), then
\[
        s(N)=N^{1/2}(1+o(1)).
\]
The upper bound \(s(N)\le N^{1/2}+O(N^{1/4})\) goes back to
Erd\H{o}s and Tur\'an~\cite{ErdosTuran41} and Lindstr\"om~\cite{Lindstrom69},
while the matching lower bound follows from Singer's construction
\cite{Singer38}. Despite this asymptotic, many natural inverse and
stability questions about dense Sidon sets remain wide open; see, for
instance, O'Bryant's annotated bibliography~\cite{OBryant04}, Gowers'
discussion of dense Sidon sets~\cite{GowersBlog}, and Eberhard and Manners' conjecture for dense Sidon sets in cyclic groups~\cite{EberhardManners}.

One robust principle is that dense Sidon sets (i.e. Sidon sets $A\subseteq [N]$ with $|A|\geq (1-o(1))N^{1/2}$) must be quite well
distributed. Erd\H{o}s and Freud~\cite{ErdosFreud91} showed that dense Sidon
sets are well distributed in short intervals; Lindstr\"om~\cite{Lindstrom98}
and Kolountzakis~\cite{Kolountzakis99} proved equidistribution in
arithmetic progressions; more recently, Ortega and
Prendiville~\cite{OrtegaPrendiville} proved equidistribution in regular
Bohr sets. There have also been studies on Sidon sets in high-dimensional lattices (see \cite{C10}) that demonstrate phenomena similar to those in the case of integers.

As a consequence, if a Sidon set $A\subseteq [N]$ is highly non-uniform, then one would expect that $|A|$ must be significantly smaller than $N^{1/2}$. One type of non-uniformity is being \emph{ill-distributed} modulo primes, meaning that $|A_p|<(1-\eps)p$ for some $\eps>0$ and all primes $p$, where 
\[A_p = \{a\pmod p:a\in A\}\subseteq \Z/p\Z.\] 
This definition comes from inverse problems for the large sieve and the larger sieve, which appears for example in the work of Helfgott and Venkatesh \cite{HelfgottVenkatesh} and Walsh \cite{WalshInvSieve, WalshAlgebraicity}.

We begin by establishing a general upper bound for ill-distributed Sidon sets. 

\begin{theorem}\label{thm:ill-distributed}
Let $\alpha \in (0,1)$, and let $A\subset [N]$ be a Sidon set such that
\[
|A_p|\le \alpha p
\]
for every prime $p$. Then for every $\delta \in (0,1/4)$,
\[
|A|\ll_{\alpha,\delta} \sqrt{N}\exp\!\left(-\left(\frac14-\delta\right)\log\frac1{\alpha}\cdot \frac{\log N}{\log\log N}\right).
\]
\end{theorem}

The most important special case for us is obtained by taking \(A\) to
be a set of perfect squares. Let \[\mathcal{S}_N:=\{1^2,2^2,\ldots,N^2\}.\] Since the squares occupy only \((p+1)/2\) residue classes modulo every
odd prime \(p\), Theorem~\ref{thm:ill-distributed} immediately gives the
following.

\begin{corollary}
\label{cor:sidon-squares}
Let \(A\subseteq \mathcal{S}_N\) be a Sidon set. Then, 
\[
        |A|
        \leq
        N
        \exp\left(
        -\left(\frac{\log 2}{2}-o(1)\right)
        \frac{\log N}{\log\log N}
        \right).
\]
\end{corollary}

This makes progress on a classical problem of Erd\H{o}s on Sidon subsets of the squares, listed
as Erd\H{o}s Problem~\#773~\cite{ErdosProblems773}. This problem was first studied by Alon and Erd\H{o}s~\cite{AlonErdos85}, who observed that the
Landau--Ramanujan theorem \cite{Landau1908} gives
\[
        |A|\ll \frac{N}{(\log N)^{1/4}}.
\]
Indeed, all sums \(a+a'\), with \(a,a'\in \mathcal{S}_N\), are integers at most
\(2N^2\) representable as sums of two squares, and there are only
\(O(N^2/\sqrt{\log N})\) such integers. Alon and Erd\H{o}s also gave a
probabilistic construction of size \(N^{2/3-o(1)}\), later refined by
Lefmann and Thiele~\cite{LefmannThiele95} to \(N^{2/3}\). Hanson
\cite{HansonLargeSieve} improved the upper bound to
\(O(N/(\log N)^{1/2})\) under a slightly more general hypothesis. No
super-polynomial-in-\(\log N\) saving was previously known.

Next, we discuss the connection between Theorem~\ref{thm:ill-distributed} and sieve theory. Under the assumption $A\subset [N]$ and $|A_p|\leq \alpha p$ for all primes $p$, Montgomery's large sieve \cite{M68} already implies \(|A|\ll_\alpha N^{1/2}\), and this is sharp (up to the implied constant) in the range
\(\alpha\ge 1/2\) \footnote{When $\alpha<1/2$, Gallagher's larger sieve~\cite{gallagher} gives $|A|\ll_{\alpha} N^{\alpha}$. So our result is only nontrivial when $\alpha\geq 1/2$.}, as shown by taking $A$ to be squares up to $N$. Green and Harper \cite{GreenHarper} conjectured that the squares, and more generally quadratic
progressions, are essentially the only sharp examples. Hanson
\cite{HansonLargeSieve} proved a soft form of this principle: a set $A\subseteq [N]$ of size $\gg \sqrt{N}$
occupying at most half the residue classes modulo each prime has large
additive correlation with the set of squares up to $N$. As a consequence, he deduced the improved bound \(O(N/(\log N)^{1/2})\) on Sidon sets in $\mathcal{S}_N$. Corollary~\ref{cor:sidon-squares} significantly improves Hanson's bound, and it is interesting to see if our techniques can be used to make progress on this inverse sieve conjecture.

We prove a more general version of Theorem~\ref{thm:ill-distributed} in Theorem~\ref{thm:sidonbounded}, where we consider high-dimensional subsets $A\subseteq [N]^\ell$ that are ill-distributed and $r_{A-A}(n)\leq g$ for some positive integer $g$ and all $n\neq 0$ (note $A$ is a Sidon set when $g=1$). Moreover, when $g\gg \log \log N$, our upper bound on $|A|$ is of the right shape, as we have a probabilistic lower bound construction.

\subsection{Two-dimensional distance problems}

We next turn to two-dimensional distance problems. As usual, we denote by $[N]^2$ the Cartesian product $\{1,\ldots,N\}\times\{1,\ldots,N\}$. The first grid problem asks for large subsets of \([N]^2\) in which no
positive distance occurs more than once. Define
\[
        g(N):=
        \max\{|A|:A\subseteq [N]^2,\ \text{no positive distance is
        determined twice by }A\}.
\]
Equivalently, no two distinct unordered pairs of points of \(A\) have the same Euclidean distance.

The problem of determining the asymptotic value of $g(N)$ as $N$ becomes large is a classical problem of Erd\H{o}s and Guy~\cite{ErdosGuy70}. The whole \(N\times N\) grid determines only
\(O(N^2/\sqrt{\log N})\) distinct distances by the Landau--Ramanujan
theorem \cite{Landau1908}, and this already gives the elementary upper bound
\[
        g(N)\ll \frac{N}{(\log N)^{1/4}}.
\]

Erd\H{o}s and Guy obtained the lower bound $g(N)> N^{2/3-c/\log\log N}$ for some $c>0$. 
Lefmann and Thiele \cite{LefmannThiele95} improved this bound to $g(N)\gg N^{2/3}$. 
We derive the first improvement for this problem in over 30 years.

\begin{theorem}
\label{thm:no-repeated-distances}
There is an absolute constant \(c>0\) such that
\[
        g(N)
        \ll
        N
        \exp\left(
        -c\frac{\log N}{\log\log N}
        \right).
\]
\end{theorem}

The second grid distance problem that we study concerns the following parameter:
\[
        f(N):=
        \max\{|A|:A\subseteq [N]^2,\ A\text{ contains no isosceles
        triangle}\}.
\]
Here and throughout the paper, isosceles triangles are allowed to be
degenerate: three equally spaced collinear points count as an isosceles
triangle. The problem of determining the asymptotic value of \(f(N)\) as $N \to \infty$ was popularized in a
MathOverflow question of Ellenberg~\cite{EllenbergMO} and has also appeared
recently in the PatternBoost work of Charton, Ellenberg, Wagner, and
Williamson~\cite{PatternBoost}, where the authors used AI-assisted search
to construct lower bounds for $f(N)$ for small values of $N$. Their computations suggest that \(f(N)\) may be much closer to linear than to quadratic.

There are two simple upper bounds for $f(N)$. 
First, fixing a point \(a\in A\), the
no-isosceles condition implies that all distances from \(a\) to
\(A\setminus\{a\}\) are distinct. Since the \(N\times N\) grid determines
only \(O(N^2/\sqrt{\log N})\) possible squared distances, this gives
\[
        f(N)\ll \frac{N^2}{\sqrt{\log N}}.
\] 

Second, because degenerate isosceles triangles are forbidden, each
vertical fibre of \(A\) is free of nontrivial three-term arithmetic
progressions. Hence $|A|\le N r_3(N)$, where \(r_3(N)\) denotes the largest size of a \(3\)-term-progression-free
subset of \([N]\). Modern Roth-type bounds therefore imply strong
subquadratic estimates for \(f(N)\). However, this argument only uses degenerate isosceles triangles. Our method uses all isosceles triangles through
the observation that, for each fixed distance, the corresponding distance
graph on \(A\) is a matching.

We derive the following stronger bound.

\begin{theorem}
\label{thm:no-isosceles}
There is an absolute constant \(c>0\) such that
\[
        f(N)
        \ll
        N^2
        \exp\left(
        -c\frac{\log N}{\log\log N}
        \right).
\]
\end{theorem}

Both grid-distance results follow from a more general theorem for binary
quadratic forms. Let $Q(x,y)=\alpha x^2+\beta xy+\gamma y^2$ be a primitive positive definite integral binary quadratic form. For
\(A\subseteq [N]^2\) and \(m\ge 1\), define
\[
        R_{A,Q}(m)
        :=
        \#\{(a,b)\in A^2:a\ne b,\ Q(a-b)=m\}.
\]

\begin{theorem}
\label{thm:quadratic-form}
Let \(Q\) be a primitive positive definite integral binary quadratic
form. Suppose \(A\subseteq [N]^2\) satisfies $R_{A,Q}(m)\le B$ for every \(m\ge 1\). Then there exists a constant \(c_Q>0\), depending
only on \(Q\), such that
\[
        |A|
        \ll_Q
        \sqrt{B}\,N
        \exp\left(
        -c_Q\frac{\log N}{\log\log N}
        \right).
\]
\end{theorem}

The proof is a two-dimensional analogue of the argument used for
squares. In the one-dimensional Sidon-in-squares problem, the key local
fact is that squares occupy only half of the residue classes modulo an
odd prime. In the two-dimensional setting, the corresponding local fact
is that, for large primes \(p\) for which the discriminant of \(Q\) is a
quadratic residue, the congruence
\[
        Q(x,y)\equiv 0\pmod p
\]
splits into two lines. 

Taking \(Q(x,y)=x^2+y^2\), Theorem~\ref{thm:quadratic-form} gives the
two aforementioned upper bounds on $f(N)$ and $g(N)$.
This is immediate for Theorem \ref{thm:no-repeated-distances}.
Indeed, if no distance is repeated, then each squared distance arises
from at most one unordered pair of points, and hence from at most two
ordered pairs.

For Theorem \ref{thm:no-isosceles}, fix a squared distance \(m\) and form the graph on
vertex set \(A\) whose edges are pairs of points at squared distance
\(m\). If two edges share a vertex, then the corresponding three points
form an isosceles triangle, possibly degenerate. Thus each fixed-distance
graph is a matching, and so it has at most \(|A|/2\) unordered edges, or
at most \(|A|\) ordered edges.  Applying
Theorem~\ref{thm:quadratic-form} with \(B=|A|\) gives the result.

\subsection{Generalized Sidon sets in perfect powers}
The next part of the paper studies a relaxation of the Sidon-in-squares
problem, as well as a related question in perfect cubes and perfect fourth powers.

A natural and well-studied generalization of Sidon sets is the family of $B_h[g]$ sets. Let $h$ be an integer with $h\geq 2$, and $A$ be a subset of integers, we define
\[
R_{A,h}(m):=\#\{a_1,a_2,\ldots, a_h\in A: \sum_{i=1}^{h} a_i=m, \ a_1\leq a_2\leq \cdots \leq a_h\}
\]
for each integer $m$. We call $A \subset \Z$ a \emph{$B_h[g]$ set} if $R_{A,h}(m)\leq g$ for all integers $m$. In particular, Sidon sets are precisely $B_2[1]$ sets. In general, determining the asymptotic of the size of the largest $B_h[g]$ set in $[N]$ remains an open problem. We refer the reader to \cite{CillerueloKissRuzsaVinuesa, CillerueloRuzsaVinuesa, JTT22} and the references therein for more discussions. 

Studying $B_h[g]$ sets in perfect $k$-th powers is known to be challenging. A well-known conjecture of Lander--Parkin--Selfridge \cite{LPS} is: if the equation \[\sum _{i=1}^{n}a_{i}^{k}=\sum _{j=1}^{m}b_{j}^{k}\] holds in positive integers with $a_i \neq b_j$ for all $1\leq i \leq n$ and $1\leq j \leq m$, then $m+n \geq k$. In particular, this conjecture implies that for $k\geq 5$, the set of perfect $k$-th powers is a Sidon set, a conjecture first formulated by Erd\H os and Graham \cite[page 53]{EG80}. There are other relevant questions, for example, Erd\H{o}s \cite{Erdos80} asked if a Sidon subset in $\{1^3, 2^3, \ldots, N^3\}$ can have size $\gg N$; see also Erd\H{o}s Problem~\#1206
\cite{ErdosProblems1206}. We refer to \cite{C10b, GK24, KS, OBryant04, V00} and the references therein for further discussion of $B_h[g]$ sets in $k$-th powers.

The following theorem concerns $B_2[g]$ sets in squares.

\begin{theorem}
\label{thm:almost-sidon-squares}
There exists an absolute constant $c>0$ such that if $A\subseteq \mathcal{S}_N$ is a $B_2[g]$-set, then
\[
        |A|
        \ll
        g^{1/4} N
        \exp\left(
        -c\frac{\log N}{\log\log N}
        \right).
\]
\end{theorem}

We note that for \(g>1\), bounded sum multiplicity no longer gives pointwise control
of differences (and vice versa), so the argument used for Theorem \ref{thm:ill-distributed} breaks. However,
a different sieve method, enhanced by entropy consideration, can still give a bound of the same shape. Again, when $g\gg \log \log N$, our upper bound is of the right shape; see Proposition~\ref{prop:growing-g-alteration}.

The set of sums of two squares has a multiplicative structure coming from $\Z[i]$ and has zero density by the Landau-Ramanujan theorem~\cite{Landau1908}. On the other hand, for $k\geq 3$, the arithmetic property of the set of sums of $k$ many $k$-th powers remains mysterious. For example, the famous Hardy-Littlewood Hypothesis K asked the asymptotic behavior of the number of representations of $n$ as the sum of $k$ many $k$-th powers, as $n \to \infty$; see also Erd\H os and Graham \cite{EG80} and Erd\H os Problem \#322 \cite{ErdosProblems322}. A well-known conjecture, based on a heuristic application of the Hardy-Littlewood circle method, predicts that when $k\geq 3$, the set of sums of $k$ many $k$-th powers has positive density \cite{DHL98}. We refer to \cite{KW99, W15} for the state of the art for sums of $3$ cubes and sums of $4$ fourth powers, and the very recent survey by Maynard~\cite{M26}. In particular, because of this conjecture, there is no known nontrivial upper bound on $B_k[g]$ subsets of $\{1^k,2^k,\ldots, N^k\}$.

Using our entropic sieve, we prove the first nontrivial upper bounds on $B_3[g]$ subsets of cubes and $B_4[g]$ subsets of fourth powers. Consider the set of first $N$ perfect cubes $\mathcal{C}_N=\{1^3,2^3,\ldots, N^3\}$, and the set of the first \(N\) fourth powers $\mathcal{Q}_N:=\{1^4,2^4,\ldots,N^4\}$.

\begin{theorem}
\label{thm:almost-sidon-cubes}
 There exists an absolute constant $c>0$ such that if $A\subseteq \mathcal{C}_N$ is a $B_3[g]$-set, then
    \[|A|\ll g^{1/9}N \exp\left(
        -c\frac{(\log N)^{1/2}}{\log\log N}
        \right).\]
\end{theorem}

\begin{theorem}
\label{thm:almost-sidon-fourth}
 There exists an absolute constant $c>0$ such that if $A\subseteq \mathcal{Q}_N$ is a $B_4[g]$-set, then
    \[|A|\ll \frac{g^{1/16}N}{(\log\log N)^{c}}.\]
\end{theorem}

The new ingredient in the proof is entropy versions of our combinatorial sieve. One version is based on Shearer's inequality. To illustrate this, we briefly outline the proof of a weaker version of Theorem~\ref{thm:almost-sidon-squares} as follows.
Write \(A=B^2=\{b^2:b\in B\}\) with \(B\subseteq [N]\). Choose many
primes \(p\equiv 1\pmod4\) near \(\log N\), and let \(q\) be a product
of \(k\) of them. If \(B\) occupies many residue classes modulo one such \(q\), then the \(2^k\) roots of \(-1\) modulo \(q\) force many pairs \((b,c)\in B^2\) with $b^2+c^2\equiv 0\pmod q$. The boundedness of \(R_{A,2}\) then gives an upper bound for these pairs.
On the other hand, if \(B\) fails to occupy many residue classes modulo
every such \(q\), then Shearer's inequality gives an entropy deficit and
forces \(B\) itself to be small. Balancing these two alternatives gives
the displayed estimate. As for Theorems~\ref{thm:almost-sidon-cubes} and \ref{thm:almost-sidon-fourth}, we can use a similar Shearer dichotomy, with additional arithmetic inputs from elliptic curves and Gauss sums. 

To establish these new results in this subsection in full, we need to employ the weighted entropy version of the combinatorial sieve. In addition to the arithmetic inputs, one of the keys is to design a suitable weight function. We refer to Theorem~\ref{thm:weight} for a precise framework. For comparison of the two entropy versions, we refer to Remark~\ref{remark:Shearer}. 

Unfortunately, our techniques do not give nontrivial bounds for $B_h[g]$ sets in perfect $k$-th powers in general; see Remark~\ref{rem:higherh}.

\subsection{Norm form extensions}
Our last main results are number-field versions of the same branching phenomenon. The case \(K=\mathbb Q(i)\) is exactly the Gaussian-prime
interpretation of the Euclidean norm \(x^2+y^2\). For a general number
field, completely split rational primes replace primes \(p\equiv1\pmod4\),
and the prime ideals above them replace the isotropic lines.

There are two parallel statements. The first concerns repeated norm
distances. Let \(K\) be a number field of degree \(r\ge2\), with ring of
integers \(\mathcal O_K\). For a finite set \(A\subseteq\mathcal O_K\), define
\[
        R_{A,K}(q):=
        \#\{(a,b)\in A^2:a\ne b, \        N_{K/\mathbb Q}(a-b)=q\}
\]
for each $q\in \Z$, where $N_{K/\mathbb Q}$ is the norm with respect to the field extension $K/\mathbb Q$.

\begin{theorem}
\label{thm:number-field}
Let \(K\) be a number field of degree \(r\ge2\). Let
\(A\subseteq\mathcal O_K\) be finite, and suppose that $R_{A,K}(q)\le B$ for every \(q\in\mathbb Z\). If $L:=\max\{|N_{K/\mathbb Q}(a-b)|: a,b\in A, a\ne b\}$, then there exists a constant \(c_K>0\), depending only on \(K\), such that
\[
        |A|
        \ll_K
        \sqrt{BL}
        \exp\left(-c_K\frac{\log L}{\log\log L}\right).
\]
\end{theorem}

After choosing a \(\mathbb Z\)-basis of a full-rank lattice in
\(\mathcal O_K\), this becomes a theorem for norm forms on boxes. We state
one convenient version. We say a homogeneous polynomial
\(F\in\mathbb Z[x_1,\ldots,x_r]\) is a \emph{norm form associated to}
\(K\) if there are \(\omega_1,\ldots,\omega_r\in\mathcal O_K\), spanning a
full-rank \(\mathbb Z\)-module, such that
\[
        F(x_1,\ldots,x_r)
        =N_{K/\mathbb Q}(x_1\omega_1+\cdots+x_r\omega_r)
\]
for all \((x_1,\ldots,x_r)\in\mathbb Z^r\). For \(A\subseteq[N]^r\), define
\[
        R_{A,F}(q):=
        \#\{(a,b)\in A^2:a\ne b, \ 
        F(a-b)=q\}.
\]

\begin{theorem}
\label{thm:norm-forms}
Let \(F\) be a norm form associated to a number field of degree \(r\ge2\).
If \(A\subseteq[N]^r\) satisfies $R_{A,F}(q)\le B$ for every \(q\in\mathbb Z\), then there exists \(c_F>0\) such that
\[
        |A|
        \ll_F
        \sqrt B\,N^{r/2}
        \exp\left(-c_F\frac{\log N}{\log\log N}\right).
\]
\end{theorem}

It is well-known that every irreducible integral binary quadratic form is a constant multiple of a norm form of some quadratic field (see \cite[Section 2.7]{BS66}). Hence Theorem~\ref{thm:norm-forms} recovers Theorem~\ref{thm:quadratic-form}. In certain senses, the bound on $|A|$ here is optimal up to the constant $c_F$; see Remark~\ref{rem:numberfield}.

\medskip

The second number-field statement is the norm-form analogue of
Theorem~\ref{thm:almost-sidon-squares}. Following the above notations, for sets \(A_1,\ldots,A_r\subseteq [N]\), define the representation function
\[
        R_{A_1,\ldots,A_r,F}(m)
        :=
        \#\{(a_1,\ldots,a_r)\in A_1\times\cdots\times A_r:
        F(a_1,\ldots,a_r)=m\}.
\]

\begin{theorem}\label{thm:norm-almost-sidon}
 Let \(F\) be a norm form associated to a number field of degree \(r\ge 2\). Let \(A_1,\ldots,A_r\subseteq [N]\). Suppose
\(R_{A_1,\ldots,A_r,F}(m)\le g\) for every \(m\in\mathbb Z\). Then there are positive constants \(c_F,C_F\) depending only on \(F\), such that
\[
        \prod_{i=1}^{r}|A_i|
        \le
        C_F g^{1/r}N^r
        \exp\left(
        -c_F\frac{\log N}{\log\log N}
        \right).
\] 
\end{theorem}

When \(K=\mathbb Q(i)\) with the fixed basis \(\{1,i\}\), the corresponding norm form is \(F(x,y)=x^2+y^2\). Thus
Theorem~\ref{thm:norm-almost-sidon} recovers
Theorem~\ref{thm:almost-sidon-squares} by setting $A_1=A_2=B$, where $A=\{b^2: b\in B\}$. 

The norm-form results also connect our argument with the classical work of
Schmidt \cite{S72} on norm form equations of more general forms. Schmidt studied equations of the form
\[
        N_{K/\mathbb Q}(\alpha_1x_1+\cdots+\alpha_rx_r)=c,
\]
where $\alpha_1,\ldots, \alpha_r$ are fixed elements in a number field $K$, and showed that their solutions are governed by lower-rank submodules and
unit families attached to subfields. His proof uses the Subspace Theorem to
force suitable solutions into finitely many proper rational subspaces, and his
later quantitative \cite{S90} work gives coefficient-independent bounds for nondegenerate
norm form equations, with divisor-type dependence on the prime factorization of
the right-hand side. Our norm-form theorems are complementary: rather than
fixing one value \(c\) and counting all solutions to \(F(\mathbf x)=c\), we take
finite product sets \(A_1\times\cdots\times A_r\subseteq[N]^r\) and assume that
every value of the norm form has a small representation multiplicity on this
product set. The split-prime branching in our proof is the finite-combinatorial
shadow of the same arithmetic structure.

\medskip

\textbf{Notation.} We follow standard notation in arithmetic combinatorics, analytic number theory, and probability. In this paper,~$p$ always denotes a prime, and $\sum_p$ and $\prod_p$ represent sums and products over all primes. We use $\F_p$ to denote the finite field with $p$ elements. We use the Vinogradov notation $\ll$; we write $X \ll Y$ or $Y\gg X$ if there is an absolute constant $C>0$ so that $|X| \leq CY$. For a positive integer $n$, $\omega(n)$ denotes the number of distinct prime factors of $n$. Given a set $A\subseteq \Z$, a positive integer $m$ and some $h\in \Z/m\Z$, we write 
\[A(h,m):= \{a\in A: a\equiv h\pmod m\}.\]
Given a random variable $X$, we use $\mathbb{E}(X)$ and $\Ent(X)$ to denote its expectation and Shannon entropy, respectively.

\medskip

\textbf{Organization of the paper.}
In Section~\ref{sec:subspace-sieve}, we prove the combinatorial subspace
large-sieve lemma and use it to prove the ill-distributed Sidon theorem (Theorem~\ref{thm:ill-distributed}), including the application to Sidon subsets of the squares.  We also record a
matching-type random deletion construction for bounded sum and difference
multiplicity.

In Section~\ref{sec:quadratic-forms} we apply the same subspace sieve to
binary quadratic forms (Theorem~\ref{thm:quadratic-form}).  This proves the repeated-distance theorem and the
no-isosceles theorem in \([N]^2\), as well as the more general theorem for
\(Q\)-distances.  We then prove the number-field version for repeated norm
distances (Theorem~\ref{thm:number-field}) and derive the corresponding norm-form theorem on boxes (Theorem~\ref{thm:norm-forms}).

In Section~\ref{sec:entropy-sieve}, we develop an entropy-enhanced large sieve.  After proving the local entropy-saving lemma and the Shearer alternative, we apply the method to prove a weaker version of Theorem~\ref{thm:almost-sidon-cubes}. We also discuss the pros and cons of the framework.

In Section~\ref{sec:weight}, we develop a weighted version of the above entropy-enhanced large sieve. After describing the general approach, we apply the method to \(B_2[g]\)-sets in the squares (Theorem~\ref{thm:almost-sidon-squares}), to \(B_3[g]\)-sets in the cubes (Theorem~\ref{thm:almost-sidon-cubes}), to \(B_4[g]\)-sets in the fourth powers (Theorem~\ref{thm:almost-sidon-fourth}), and to norm
forms (Theorem~\ref{thm:norm-almost-sidon}). Since Theorem~\ref{thm:norm-almost-sidon} implies Theorem~\ref{thm:almost-sidon-squares} immediately, we omit the proof of Theorem~\ref{thm:almost-sidon-squares}.

In Section~\ref{sec:conclusion}, we conclude with some further remarks and directions for future research.

\section{A combinatorial subspace large-sieve and its applications}
\label{sec:subspace-sieve}

In this section, we develop a combinatorial subspace large-sieve and discuss its applications. 

Broadly speaking, given a set $A\subseteq [N]^r$, our goal is to provide upper bounds on $|A|$ using information about its difference set $A-A$. The consideration of the difference set is natural for our applications to Sidon sets and repeated distances. 

As in the proof of the classical large/larger sieve, the key step in our proof of the new combinatorial subspace large-sieve is to estimate 
\[\sum_{\substack{a,b\in A\\a\neq b}} W(a-b)\]
from above and below using local information of $A$, where $W$ is some carefully chosen weight function. Ideally we want to bound the weight function $W$ effectively by some nicely behaved arithmetic functions, for example, multiplicative functions, so that we can apply tools from analytic number theory.

Next, we introduce the precise setting. Let $A\subset [N]^r$, $d\geq 1$, and $\Lambda$ be a
finite index set. For each $\lambda\in \Lambda$, let $L_\lambda\leq (\mathbb Z/d\mathbb Z)^r$ be an additive subgroup. We write $a\in L_\lambda \pmod d$ to mean that $a \in L_\lambda+d\Z^r.$ For each $\lambda\in \Lambda$, let
\[M_\lambda:= \#\{\overline{a}+L_\lambda\in (\Z/d\Z)^r/L_\lambda: a\in A\},\]
where $\overline{a} = a\pmod d$; equivalently, $M_
\lambda$ denotes the number of cosets of $L_\lambda$ in
$(\mathbb Z/d\mathbb Z)^r$ that meet the image of $A$ modulo $d$.
Assume that $M_\lambda\leq M$
for every $\lambda\in \Lambda$.

For each $h\in \Z^r$, define the weight
\[
        W(h):=\#\{\lambda\in \Lambda: h\in L_\lambda \pmod d\}.
\]

Let $\mathcal Y$ be a set, and consider the function $F:\mathbb Z^r\setminus\{0\}\to \mathcal Y.$
For $y\in
\mathcal Y$, define the representation function
\[
        R_F(y):=\#\{(a,b)\in A^2:a\neq b,\ F(a-b)=y\}.
\]
Suppose that 
$$R_F(y)\leq B$$ for every $y\in \mathcal Y$. In our applications, this function $F$ is given by either the identity function, a quadratic form, or a norm form, and the Sidon-like property of our set $A$ implies the boundedness of the representation function $R_F$.

Finally, to bound $W$ effectively, suppose that there is a nonnegative weight function 
$        v:\mathcal Y\to \mathbb R_{\geq 0}$
such that
\[
        W(h)\leq v(F(h))
\]
for every nonzero $h\in A-A$. Set
\[
U:=\sum_{y\in F((A-A)\setminus\{0\})} v(y).
\]

\begin{lemma}[Combinatorial subspace large-sieve]\label{lem:subspacesieve}
With the notation above, if $|A|\geq 2M$, then
\[
        |A|^2\leq \frac{2MBU}{|\Lambda|}.
\]
\end{lemma}

\begin{proof}
Fix $\lambda\in\Lambda$. Partition $A$ according to its cosets modulo
$L_\lambda$. Let the nonempty parts have sizes $n_1,\dots,n_s$, where $s\leq M_\lambda$. 
Hence, by Cauchy--Schwarz and the assumption $|A|\geq 2M\geq 2M_{\lambda}$, we have 
\begin{align*}
\#\{(a,b)\in A^2:a\neq b,\ a-b\in L_\lambda \pmod d\}
&=
\sum_{j=1}^s n_j(n_j-1)\\
&=
\sum_{j=1}^s n_j^2 \ -|A|
\ge
\frac{|A|^2}{M}-|A|
\ge
\frac{|A|^2}{2M}.
\end{align*}

Summing the above estimate over $\lambda\in\Lambda$, we get
\[
\begin{aligned}
\sum_{a\neq b} W(a-b)
=
\sum_{\lambda\in\Lambda}
\#\{(a,b)\in A^2:a\neq b,\ a-b\in L_\lambda \pmod d\}      
\geq
\frac{|A|^2|\Lambda|}{2M}.
\end{aligned}
\]
On the other hand, since $W(h)\leq v(F(h))$ for all nonzero
$h\in A-A$, we have
\[
\begin{aligned}
\sum_{a\neq b} W(a-b)\leq\sum_{a\neq b}v(F(a-b))=\sum_{y\in F((A-A)\setminus \{0\})}v(y)R_F(y) 
\leq B\sum_{y\in F((A-A)\setminus \{0\})}v(y)=BU.
\end{aligned}
\]
Combining the above two estimates, we obtain
\[
        \frac{|A|^2|\Lambda|}{2M}
        \leq BU,
\]
as required.
\end{proof}

We remark that the same framework has direct analogues over function fields as well as number fields (with irreducibles/prime ideals replacing primes). Thus, with minimal changes, we can prove function fields/number fields analogues of the results in this section.

\subsection{Application I: ill-distributed sets with bounded differences}

We first apply Lemma~\ref{lem:subspacesieve} in its simplest form, where
\(\Lambda\) consists of one subgroup, namely the zero subgroup modulo a
squarefree modulus \(d\).  This recovers the large-sieve intuition in a very
concrete way.

Let \(A\subseteq[N]^r\).  We say that \(A\) is \emph{ill-distributed modulo
primes} if there is a constant \(\alpha<1\) such that
\[
        |A_p|\le \alpha p^r
\]
for every prime \(p\), where \(A_p\) denotes the image of \(A\) in
\((\mathbb Z/p\mathbb Z)^r\).  This is the natural higher-dimensional analogue
of the condition studied in inverse problems for the large sieve.

We also allow a bounded-difference multiplicity parameter.  For
\(h\in\mathbb Z^r\), write
\[
        r_{A-A}(h):=\#\{(a,b)\in A^2:a-b=h\}.
\]
The case \(r_{A-A}(h)\le1\) for every nonzero \(h\) is the difference form of
the Sidon property.  More generally, sets with \(r_{A-A}(h)\le g\) are often
called \(g\)-Golomb rulers or \(B_2^-[g]\)-sets.

\begin{theorem}\label{thm:sidonbounded}
Let $\alpha\in(0,1)$ and $g\geq 1$. Let $A\subset [N]^r$ such that $r_{A-A}(n)\leq g$ for all $n\in \Z^r \setminus \{0\}$. Suppose that
$|A_p|\leq \alpha p^r$ for every prime $p$. Then, for every
$\delta\in(0,1/4)$,
\[
        |A|
        \ll_{\alpha,\delta,r}
        \sqrt g\, N^{r/2}
        \exp\left(
        -\left(\frac14-\delta\right)
        \log\frac1\alpha
        \frac{\log N}{\log\log N}
        \right).
\]
\end{theorem}

\begin{proof}
Fix $\delta\in(0,1/4)$ and put $\lambda:=1/2-2\delta$ and $t:=\left\lfloor \lambda \log N / \log\log N \right\rfloor$. 
Let $p_1<\cdots<p_t$ be the smallest $t$ primes and set
$d:=p_1\cdots p_t$. By the prime number theorem, $d\leq N^{\lambda+o(1)}$.

We apply Lemma~\ref{lem:subspacesieve} with $\Lambda$ a singleton and
$L=\{0\}\leq(\mathbb Z/d\mathbb Z)^r$. Thus $a-b\bmod d\in L$ means
$a\equiv b\pmod d$. By the Chinese remainder theorem, $|A_d|\leq \alpha^t d^r.$
Take $F(h)=h$ and $M=\alpha^t d^r$. We can assume $|A|\geq 2M$ for otherwise we are done. Since $A$ has $g$-bounded differences, we have
$R_F(h)=r_{A-A}(h)\leq g$ for every nonzero $h$, so we can take $B=g$.

Also $W(h)=1_{d\mid h}$, and we take $v(h)=1_{d\mid h}$. Hence
\[
        U
        \leq
        \#\{h\in[-N,N]^r:d\mid h\}
        \ll_r \left(\frac Nd\right)^r.
\]

Lemma~\ref{lem:subspacesieve} gives
\[
        |A|^2
        \ll_r
        \alpha^t d^r\cdot g\left(\frac Nd\right)^r
        =
        g\alpha^t N^r.
\]
Therefore
\[
        |A|\ll_r \sqrt g\,\alpha^{t/2}N^{r/2}
        \ll_{\alpha,\delta,r}
        \sqrt g\, N^{r/2}
        \exp\left(
        -\left(\frac14-\delta\right)
        \log\frac1\alpha
        \frac{\log N}{\log\log N}
        \right),
\]
as desired.
\end{proof}

Next we apply Theorem~\ref{thm:sidonbounded} to subsets of $\mathcal{S}_N$ with bounded difference multiplicity.

\begin{corollary}\label{cor:Sidonsquare}
Let \(A\subseteq \mathcal{S}_N\) such that $r_{A-A}(n)\leq g$ for all $n\neq 0$. Then, for every \(\delta>0\),
\[
        |A|
        \ll_{\delta} \sqrt{g}
        N
        \exp\left(
        -\left(\frac{\log 2}{2}-\delta\right)
        \frac{\log N}{\log\log N}
        \right).
\]
\end{corollary}
\begin{proof}
Put \(M=N^2\), so that \(A\subseteq [M]\). Since the squares occupy at most \((p+1)/2\) residue classes modulo every odd prime \(p\), for every \(\varepsilon>0\) there is \(p_0=p_0(\varepsilon)\) such that \(|A_p|\le (1/2+\varepsilon)p\) for all primes \(p\ge p_0\). Let \(T=\prod_{p<p_0}p\), and write \(A=\bigcup_{i=0}^{T-1}B_i\), where \(B_i=\{a\in A:a\equiv i\pmod T\}\). Then each \(B_i\) occupies at most one residue class modulo every prime \(p<p_0\), and at most \((1/2+\varepsilon)p\) residue classes modulo every prime \(p\ge p_0\). Also \(r_{B_i-B_i}(n)\le g\) for all \(n\ne0\). Applying Theorem~\ref{thm:sidonbounded} to each \(B_i\), with ambient length \(M=N^2\) and \(\alpha=1/2+\varepsilon\), and using \(M^{1/2}=N\) and
\[
        \frac{\log M}{\log\log M}
        =
        (2+o(1))\frac{\log N}{\log\log N},
\]
we obtain the desired bound after summing over the \(O_\varepsilon(1)\) sets \(B_i\) and choosing \(\varepsilon\) and the parameter in Theorem~\ref{thm:sidonbounded} sufficiently small in terms of \(\delta\).
\end{proof}

The next proposition shows that the dependence on \(g\) in
Corollary~\ref{cor:Sidonsquare} is of the right general shape once \(g\) is allowed to
grow as a function of $N$.  The construction is a standard random deletion argument, but the
relevant obstruction counts use the divisor-function bounds for sums and
differences of two squares.  In particular, for
\[
        \log\log N\ll g\ll
        \exp\left(O\left(\frac{\log N}{\log\log N}\right)\right),
\]
the lower bound below matches the upper bound from
Corollary~\ref{cor:Sidonsquare} up to the constant in the exponential.

\begin{proposition}\label{prop:growing-g-alteration}
Let \(g=g(N)\ge 1\) be integer-valued. Suppose that
\begin{equation*}
        g\le
        \exp\left((2\log 2+o(1))\frac{\log N}{\log\log N}\right).
\end{equation*}
Then there exists \(A\subset \mathcal{S}_N\) such that $R_{A,2}(n)\le g$ for all $n$ and $r_{A-A}(n)\le g$ for all $n\neq 0,$ and 
\begin{equation}\label{eq:growing-g-lower-bound}
        |A|
        \ge
        N\exp\left(
        -\frac{\log N}{2g+1}
        -\left(\frac{2g}{2g+1}\log 2+o(1)\right)
        \frac{\log N}{\log\log N}
        +\frac{g}{2g+1}\log g
        \right).
\end{equation}
Consequently, if \(g=c\log\log N\), where \(c>0\) is fixed, then
\begin{equation*}
        |A|
        \ge
        N\exp\left(
        -\left(\frac1{2c}+\log 2+o(1)\right)
        \frac{\log N}{\log\log N}
        \right),
\end{equation*}
and if \(\log\log N\ll g\), then
\begin{equation*}
        |A|
        \ge
        \sqrt g\, N
        \exp\left(
        -(\log 2+o(1))\frac{\log N}{\log\log N}
        \right).
\end{equation*}
\end{proposition}
\begin{proof}
Let \(L_N=\log N/\log\log N\). By the standard divisor bound, uniformly in
\(n\) and \(m\ne0\),
\[
        r_{\mathcal S_N+\mathcal S_N}(n),\,
        r_{\mathcal S_N-\mathcal S_N}(m)
        \le \exp((2\log2+o(1))L_N).
\]
Choose an integer \(M\le \exp((2\log2+o(1))L_N)\) dominating both representation
functions. If \(g\ge M\), then \(A=\mathcal S_N\) works, so assume \(g<M\). Put
\(D=\binom Mg\). Choose \(\xi=\xi(N)\to\infty\) with
\(\xi=\exp(o(L_N))\), and form a random subset \(A_0\subseteq \mathcal S_N\) by
retaining each element independently with probability
\[
        \rho=\xi^{-1}(ND)^{-1/(2g+1)}.
\]

Call a collection of \(g+1\) distinct representations of the same integer a
bad cluster. For \(0\le r\le M\), we use the elementary bound
\[
        \binom r{g+1}\le \frac{r}{g+1}\binom Mg=\frac{rD}{g+1}.
\]
The expected number of non-diagonal bad sum-clusters is
\(\ll \rho^{2g+2}N^2D=o(\rho N)\), and bad sum-clusters containing a diagonal
representation contribute \(\ll ND\rho^{2g+1}=o(\rho N)\). The same estimate
also gives \(o(\rho N)\) for bad difference-clusters whose \(g+1\) representations
are pairwise disjoint.

It remains to consider overlapping difference-clusters. For a fixed nonzero
difference \(m\), view the representations of \(m\) as edges on
\(\mathcal S_N\). Each component has at most two edges, since otherwise four
distinct squares would form an arithmetic progression. Thus overlaps are
two-edge components, equivalently three-term arithmetic progressions of
squares; the total number of these is \(N^{1+o(1)}\). If an overlapping cluster
contains \(q\) such two-edge components, with \(1\le q\le (g+1)/2\), then it
uses \(2g+2-q\) distinct elements, and the number of such clusters is at most
\[
        N^{1+o(1)}\binom M{q-1}\binom M{g+1-2q}.
\]
Using
\[
        \binom M{q-1}\binom M{g+1-2q}
        \le C^gD^{1-q/(2g+1)}
\]
for an absolute constant \(C\), the expected contribution of these clusters,
divided by \(\rho N\), is at most
\[
        N^{o(1)}\rho^{2g+1-q}C^gD^{1-q/(2g+1)}
        =
        N^{o(1)}C^g\xi^{-(2g+1-q)}N^{-1+q/(2g+1)}
        =
        o(g^{-1}),
\]
uniformly in \(q\). Summing over \(q\), the expected number of all bad clusters
is \(o(\rho N)\).

Delete one element from each bad cluster in \(A_0\), and call the remaining set
\(A\). Then some choice satisfies \(|A|\ge (1-o(1))\rho N\), and by construction
\(R_{A,2}(n)\le g\) for all \(n\) and \(r_{A-A}(m)\le g\) for all \(m\ne0\).
Finally, since \(D=\binom Mg\le (eM/g)^g\), \(\log \xi=o(L_N)\), and
\(\log M=(2\log2+o(1))L_N\), we get
\[
        |A|
        \ge
        N\exp\left(
        -\frac{\log N}{2g+1}
        -\left(\frac{2g}{2g+1}\log2+o(1)\right)L_N
        +\frac{g}{2g+1}\log g
        \right).
\]
This proves estimate~\eqref{eq:growing-g-lower-bound}, and the two stated consequences
follow immediately by substituting \(g=c\log\log N\) and
\(\log\log N\ll g\), respectively.
\end{proof}

The following remark shows the sharpness of Theorem~\ref{thm:sidonbounded} from another perspective.

\begin{remark}
Theorem~\ref{thm:sidonbounded} has the following immediate corollary: if $A\subseteq [N]$ satisfies $|A|\gg \sqrt{N}$ and $|A_p|\leq p/2$ for every prime $p$, then there is $n \in [1,N]$ such that
\begin{equation}\label{eq:rA-A}
r_{A-A}(n)\geq \exp\bigg(\bigg(\frac{\log 2}{4}-o(1)\bigg)\frac{\log N}{\log\log N}\bigg).
\end{equation}

Let $S$ be the set of squares up to $N$. From the estimate for the divisor function we know that \[\max_{n\geq 1}r_{S-S}(n) \asymp \exp(\log 2 \cdot \log N/\log\log N).\]
    So the estimate~\eqref{eq:rA-A} is sharp up to constant in the exponential.
\end{remark}

\subsection{Application II: repeated distances and isosceles triangles in \([N]^2\)}
\label{sec:quadratic-forms}

In this subsection, we prove Theorem~\ref{thm:quadratic-form}. As mentioned in the introduction, Theorem~\ref{thm:quadratic-form} is a special case of Theorem~\ref{thm:norm-forms}. Nevertheless, here we present a self-contained proof of Theorem~\ref{thm:quadratic-form} as it highlights how the local geometry (isotropic line systems) allows us to carry out the sieve-type argument. Moreover, it inspires the combinatorial subspace sieve lemma and the proof of Theorem~\ref{thm:number-field}.

\begin{proof}[Proof of Theorem~\ref{thm:quadratic-form}]
Let
\[
        Q(x,y)=\alpha x^2+\beta xy+\gamma y^2.
\]
Let $c_0>0$ be a sufficiently small constant depending only on $Q$, and put
\[
        t:=\left\lfloor c_0\frac{\log N}{\log\log N}\right\rfloor.
\]
Let $\Delta=\beta^2-4\alpha\gamma$ be the discriminant of $Q$. Since $Q$ is positive
definite, $\Delta<0$, and in particular $\Delta$ is nonsquare. By
quadratic reciprocity, the condition that $\Delta$ is a quadratic residue
modulo a prime $p$ is, away from the finitely many primes dividing $2\Delta$,
equivalent to requiring $p$ to lie in a fixed union of reduced residue
classes modulo an integer depending only on $\Delta$. Hence, by the prime
number theorem in arithmetic progressions, there are $\gg_Q X/\log X$
primes $p\in [X,C_QX]$ such that $p\nmid 2\alpha\Delta$ and $\Delta$ is a
quadratic residue modulo $p$, provided $C_Q>1$ is chosen sufficiently
large. Applying this with $X=\log N$, and taking $c_0>0$ sufficiently small, we
can choose distinct primes $p_1,\dots,p_t$ in the interval
$\log N\le p_i\le C_Q\log N$ such that $p_i\nmid 2\alpha\Delta$ and $\Delta$
is a quadratic residue modulo $p_i$. Set $d:=p_1\cdots p_t$. Then
\[
        d\le (C_Q\log N)^t=N^{c_0+o(1)}.
\]

For each $i$, since $\Delta$ is a quadratic residue modulo $p_i$ and
$p_i\nmid 2\alpha\Delta$, the quadratic polynomial $\alpha T^2+\beta T+\gamma$ has two distinct roots modulo $p_i$. Denote them by $\rho_{i,+}$ and
$\rho_{i,-}$. Equivalently,
\[
         \alpha T^2+\beta T+\gamma\equiv \alpha(T-\rho_{i,+})(T-\rho_{i,-})
        \pmod {p_i}.
\]
Since $p_i\nmid \alpha$, the congruence $Q(x,y)\equiv0\pmod {p_i}$ is therefore
the union of the two distinct lines
\[
        x\equiv \rho_{i,+}y\pmod {p_i},
        \qquad
        x\equiv \rho_{i,-}y\pmod {p_i}.
\]
For each sign vector $\sigma=(\sigma_1,\dots,\sigma_t)\in\{\pm1\}^t$, define
\[
        L_\sigma
        :=
        \{(x,y)\in(\mathbb Z/d\mathbb Z)^2:
        x\equiv \rho_{i,\sigma_i}y\pmod {p_i}
        \text{ for every }i\}.
\]
Each $L_\sigma$ is an additive subgroup of $(\mathbb Z/d\mathbb Z)^2$, since it is defined by homogeneous linear congruences. Moreover, by the Chinese remainder theorem, $|L_\sigma|=d$, so $L_\sigma$ has index $d$.

We apply Lemma~\ref{lem:subspacesieve} with $\Lambda=\{\pm1\}^t$, so $|\Lambda|=2^t$, and with $M=d$. Thus, we can assume $|A|\ge 2M$ for otherwise we are done. For $h=(x,y)\in \Z^2$, define
\[
        W(h):=\#\{\sigma\in\{\pm1\}^t:(x,y)\bmod d\in L_\sigma\}.
\]
For $1\leq i \leq t$, let $w_i(h)$ be the number of local lines
$x\equiv \rho_{i,+}y\pmod {p_i}$ and
$x\equiv \rho_{i,-}y\pmod {p_i}$ containing $h\bmod p_i$.
Then $W(h)=\prod_{i=1}^t w_i(h)$. If $w_i(h)>0$, then
$Q(h)\equiv0\pmod {p_i}$. Hence $W(h)>0$ implies $d\mid Q(h)$.
Moreover, $w_i(h)=2$ only if $h$ lies on both local lines. Since the two
roots $\rho_{i,+}$ and $\rho_{i,-}$ are distinct, this forces
$x\equiv y\equiv0\pmod {p_i}$, and hence $p_i^2\mid Q(h)$. Therefore,
writing $m=Q(h)$,
\[
        W(h)
        \le
        1_{d\mid m}\,2^{\omega(\gcd(m/d,d))}.
\]
Thus we can take
\[
        v(m):=1_{d\mid m}\,2^{\omega(\gcd(m/d,d))}.
\]

Since $Q$ is fixed and positive definite, every relevant $m=Q(u,v)$ with $u,v\in[-N,N]$ satisfies $m\ll_Q N^2$. Hence
\[
        U
        :=
        \sum_{m\in F((A-A)\setminus\{0\})} v(m)
        \le
        \sum_{m\ll_Q N^2/d}2^{\omega(\gcd(m,d))}=\sum_{m\ll_Q N^2/d} \sum_{\substack{e\mid d\\ e\mid m}}1.
\]
since $d$ is squarefree. Therefore
\[
U\ll_Q
\sum_{e\mid d}\left(\frac{N^2}{de}+1\right)
=
\frac{N^2}{d}\prod_{i=1}^t\left(1+\frac1{p_i}\right)+2^t
\ll_Q \frac{N^2}{d}.
\]

Lemma~\ref{lem:subspacesieve} gives
\[
        |A|^2
        \ll_Q
        \frac{d\cdot B\cdot (N^2/d)}{2^t}
        =
        BN^2 2^{-t}.
\]
Since $t=\lfloor c_0\log N/\log\log N\rfloor$, this gives
\[
        |A|^2
        \ll_Q
        BN^2
        \exp\left(-c_Q\frac{\log N}{\log\log N}\right)
\]
for some constant $c_Q>0$ depending only on $Q$, as required.
\end{proof}

\begin{remark}
The assumptions that \(Q\) be primitive, positive definite, and integral are mainly for notational convenience. After clearing denominators and excluding finitely many primes, the same argument applies to any nonzero homogeneous binary quadratic form \(Q\) of rank \(2\), provided
\[
        \#\{(a,b)\in A^2:a\ne b,\ Q(a-b)=m\}\le B
\]
for every \(m\), including \(m=0\). Indeed, for a positive-density set of primes, the reduction of \(Q\) splits into two distinct linear factors. The rank-one case \(Q=L^2\) is easier and follows from applying the elementary difference-counting bound to the projection \(L(A)\). Thus in all cases one obtains
\[
        |A|\ll_Q \sqrt B\,N
        \exp\!\left(-c_Q\frac{\log N}{\log\log N}\right).
\]
\end{remark}

\subsection{Application III: a number field extension}
In this subsection, we prove Theorem~\ref{thm:number-field} and Theorem~\ref{thm:norm-forms}. To begin with, we recall some notations from algebraic number theory. Given a number field $K$ and its ring of integers $\cO_K$, we use $\mathrm{Spec}(\cO_K)$ to denote the set of prime ideals in $\cO_K$. For any nonzero ideal $I\lhd \cO_K$, we use $\mathrm{Nm}(I)$ to denote its absolute norm, i.e. $\mathrm{Nm}(I) = |\cO_K/I|$. For more properties on the norm of elements and ideals in $\cO_K$, and the statement of Chebotarev density theorem, we refer the reader to \cite{BS66, NeukirchANT}.

\begin{proof}[Proof of Theorem~\ref{thm:number-field}]
    Let $c_0>0$ be a sufficiently small constant depending only on $K$, and put $t:=\left\lfloor c_0\frac{\log L}{\log\log L}\right\rfloor$. By the Chebotarev density theorem, there are $\gg_K X/\log X$ many rational primes $p\in [X,C_KX]$ such that $p$ splits completely in $\cO_K$, provided $C_K>1$ is chosen sufficiently large. Applying this with $X = \log L$, and taking $c_0>0$ sufficiently small, we can choose distinct rational primes $p_1,\ldots,p_t\in [\log L, C_K\log L]$ that split completely in $\cO_K$. Let $d = p_1\cdots p_t$. Then $d\leq (C_K\log L)^t = L^{c_0+o(1)}$. For each $1\leq i\leq t$, since $p_i$ splits completely, we must have $(p_i) = \mathfrak{P}_{i,1}\cdots\mathfrak{P}_{i,r}$ for distinct prime ideals $\mathfrak{P}_{i,j}\in \mathrm{Spec}(\cO_K)$. By the Chinese remainder theorem, 
    \[\cO_K/(p_i)\cong \prod_{j=1}^r \cO_K/\mathfrak{P}_{i,j}\cong \Z_{p_i}^r,\]
    which implies 
    \begin{equation}\label{eq:iso}
    \cO_K/(d)\cong \prod_{i=1}^t \cO_K/(p_i)\cong \prod_{i=1}^t\prod_{j=1}^r \cO_K/\mathfrak{P}_{i,j}\cong \Z_d^r.
    \end{equation}
    For each vector $\sigma\in [r]^t$, define 
    \[L_\sigma = \prod_{i=1}^t\mathfrak{P}_{i,\sigma_i},\]
    which can naturally be viewed as a subgroup of $\Z_d^r$ via the isomorphism in equation~\eqref{eq:iso}. In particular, $[\cO_K:L_\sigma] = d$.

    Next, we apply Lemma~\ref{lem:subspacesieve} with $\Lambda=[r]^t$ and with $M=d$. We can assume $|A|\ge 2M$, for otherwise we are done. For any $a\in \cO_K$, define 
    \[W(a):= \#\{\sigma\in [r]^t: a\in L_\sigma\}.\]
    For $1\leq i\leq t$, let $w_i(a) = \#\{1\leq j\leq r: a\in \mathfrak{P}_{i,j}\}$. Then $W(a) = \prod_{i=1}^t w_i(a)$. If $w_i(a)>0$, then $p_i\mid N_{K/\Q}(a)$ since $\mathrm{Nm}(\mathfrak{P}_{i,j}) = p_i$ for all $j$. Hence $W(a)>0$ implies $d\mid N_{K/\Q}(a)$. Moreover, $w_i(a)\geq 2$ only if $a\in \mathfrak{P}_{i,j}\cap\mathfrak{P}_{i,k}$ for some $j\neq k$, which forces $\mathrm{Nm}(\mathfrak{P}_{i,j}\mathfrak{P}_{i,k})\mid N_{K/\Q}(a)$ and hence $p_i^2\mid N_{K/\Q}(a)$. As a consequence, writing $q = N_{K/\Q}(a)$, we get
    \[
    W(a)\leq 1_{d\mid q}\cdot r^{\omega(\gcd(q/d,d))}.
    \]
    Thus we can take 
    \[
    v(q) := 1_{d\mid q}\cdot r^{\omega(\gcd(q/d,d))}.
    \]

    From the definition of $L$, we have
    \begin{align*}
    U:&= \sum_{q\in N_{K/\Q}((A-A)\setminus\{0\})} v(q)\leq \sum_{|q|\leq L} 1_{d\mid q}\cdot r^{\omega(\gcd(q/d,d))}\\
    &\ll \sum_{1\leq m\leq L/d} r^{\omega(\gcd(m,d))} = \sum_{1\leq m\leq L/d}\sum_{\substack{e\mid d\\e\mid m}}(r-1)^{\omega(e)}
    \end{align*}
    since $d$ is squarefree. Therefore
    \[
    U\ll \sum_{e\mid d}(r-1)^{\omega(e)}\frac{L}{de} = \frac{L}{d}\prod_{i=1}^t\left(1+\frac{r-1}{p_i}\right)\ll_K \frac{L}{d}.
    \]

    Now Lemma~\ref{lem:subspacesieve} yields
    \[
    |A|^2\ll_K \frac{d\cdot B\cdot (L/d)}{r^t} = BLr^{-t}.
    \]
    Since $t = \lfloor c_0\log L/\log \log L\rfloor$, this gives
    \[
    |A|\ll_K \sqrt{BL}\exp\left(-c_K\frac{\log L}{\log\log L}\right)
    \]
    for some constant $c_K>0$ depending only on $K$, as required.
\end{proof}

\begin{proof}[Proof of Theorem~\ref{thm:norm-forms}]
    Let $\psi: \Z^r\rightarrow \cO_K, (x_1,\ldots,x_r)\mapsto \sum_{i}x_i\omega_i$ be the standard homomorphism. Since $\{\omega_1,\ldots,\omega_r\}$ is a basis of $K$, $\psi$ must be injective. Let $A' = \psi(A)$. From the definition of $F$, we get
    \[
    F(a-b) = N_{K/\Q}(\psi(a)-\psi(b))
    \]
    for all $a,b\in A$. Moreover, since $F$ is a homogeneous polynomial of degree $r$, we must have 
    \[
    \max\{|N_{K/\Q}(\alpha)|: \alpha\in (A'-A')\setminus\{0\}\}\ll_F N^r.
    \]
    Now it follows from Theorem~\ref{thm:number-field} that 
    \[
    |A| = |A'|\ll_F \sqrt{B}N^{r/2}\exp\left(-c_F\frac{\log N}{\log\log N}\right).\qedhere
    \]
\end{proof}

\begin{remark}\label{rem:numberfield}
Let \(K\) be a number field of degree \(r\), and let
\(\{\omega_1,\ldots,\omega_r\}\) be a fixed integral basis of \(\mathcal O_K\).
Let \(F\) be the associated norm form. Theorem~\ref{thm:norm-forms} implies
that \(N^{r/2}\) points already force a norm-distance to be repeated many times:
if \(A\subset [N]^r\) and \(|A|\ge N^{r/2}\), then there is a nonzero integer
\(q\) such that
\[
        R_{A,F}(q)
        \gg_F
        \exp\left(c_F\frac{\log N}{\log\log N}\right)
\]
for some \(c_F>0\).

This is sharp up to the constant in the exponent. Identifying \(\mathbb Z^r\)
with \(\mathcal O_K=\mathbb Z\omega_1\oplus\cdots\oplus\mathbb Z\omega_r\),
consider
\[
        A_0=\{z^2:z\in\mathcal O_K\cap[-c_K\sqrt N,c_K\sqrt N]^r\},
\]
where \(c_K>0\) is sufficiently small. After a harmless translation and replacing
\(N\) by a constant multiple, this gives a subset of \([N]^r\) of size
\(\gg_K N^{r/2}\).

We claim that
\(R_{A_0,F}(q)\ll_K \exp(C_K\log N/\log\log N)\) for every nonzero integer
\(q\). Indeed, if \(N_{K/\mathbb Q}(z_1^2-z_2^2)=q\), then
\[
        N_{K/\mathbb Q}(z_1+z_2)N_{K/\mathbb Q}(z_1-z_2)=q.
\]
Thus, up to signs, the two factors have prescribed absolute norms \(q_1,q_2\)
with \(q_1q_2=|q|\). Since \(|q|\ll_K N^r\), the divisor bound gives only
\(\exp(O_K(\log N/\log\log N))\) possible choices for \(q_1,q_2\).

It remains to bound the number of elements of bounded height and prescribed
norm. For fixed \(m\ll_K N^r\), the number of ideals of \(\mathcal O_K\) of
norm \(m\) is at most \(\tau(m)^r\), where \(\tau\) is the divisor function
(see, for example, \cite[Chapter 7]{BS66}). Hence the number of relevant
principal ideals is \(\exp(O_K(\log N/\log\log N))\). For each such ideal,
choose one generator \(\alpha\) lying in the box, if it exists. Any other
generator in the box is \(u\alpha\) for some unit \(u\in\mathcal O_K^\times\),
and such units have height at most \(N^{O_K(1)}\). By the standard unit-counting
bound \cite[Chapter 3, Theorem 5.2]{L83}, there are only
\(O_K((\log N)^{\rho_K})\) such units, where \(\rho_K\) is the unit rank of \(K\).
Thus, uniformly for \(m\ll_K N^r\),
\[
        \#\{a\in\mathcal O_K\cap[-C_K\sqrt N,C_K\sqrt N]^r:
        |N_{K/\mathbb Q}(a)|=m\}
        \le
        \exp\left(O_K\left(\frac{\log N}{\log\log N}\right)\right).
\]
Applying this to \(z_1+z_2\) and \(z_1-z_2\), and summing over the possible
factorizations \(q_1q_2=|q|\), gives the claimed upper bound for
\(R_{A_0,F}(q)\).
\end{remark}

\section{An entropy-enhanced large-sieve and its applications}
\label{sec:entropy-sieve}

In this section, we develop an entropy-enhanced large-sieve. The main ingredient of our proof is to combine sieve methods and tools from entropy. Additionally, we need appropriate arithmetic inputs. 

Our new entropy-enhanced large-sieve is based on a ``spread-versus-entropy" dichotomy. If the set is sufficiently spread out modulo a suitable product of primes, then the local surplus of solutions to the relevant congruence forces many algebraic representations, contradicting the bounded algebraic multiplicity hypothesis. If the set is not spread out, then its residue distribution has a definite entropy deficit, and this entropy loss itself gives a saving in the size of the set. While this entropy-enhanced large-sieve is not strong enough for our applications, we believe it is flexible and it would have many potential applications. Moreover, the idea in this section inspired us to find a stronger weighted version that will be discussed in the next section.

\subsection{Entropy saving} We briefly recall the notion of Shannon entropy. If \(X\) is a random variable taking
values in a finite set \(\Omega\), its entropy is defined by
\[
        \Ent(X):=-\sum_{\omega\in\Omega}
        \mathbb P(X=\omega)\log \mathbb P(X=\omega),
\]
where we use the convention \(0\log 0=0\). We shall use two elementary facts. First, if \(X\) is supported on a set of size
at most \(m\), then \(\Ent(X)\le \log m\), with equality only when \(X\) is uniformly
distributed on a set of size \(m\). Second, if \(X_1,\ldots,X_s\) are independent
random variables, then $\Ent(X_1,\ldots,X_s)=\sum_{i=1}^{s}\Ent(X_i).$

We shall also use Shearer's inequality~\cite{Shearer} in the following form. Let
\(Z=(Z_1,\ldots,Z_\ell)\) be a random vector, and let \(\mathcal F\) be a family
of subsets of \([\ell]\). Suppose that every coordinate \(j\in[\ell]\) belongs to
at least \(d\) members of \(\mathcal F\). For \(J\subseteq[\ell]\), write
\(Z_J=(Z_j)_{j\in J}\). Then
\[
        \Ent(Z)\le \frac1d\sum_{J\in\mathcal F}\Ent(Z_J).
\]

We first prove a local entropy-saving lemma.

\begin{lemma}\label{lem:local-entropy-delta}
Let \(Y\) be a random variable taking values in \(\mathbb Z/q\mathbb Z\), and let
\(0<\delta<1\). Suppose that at most \((1-\delta)q\) residue classes \(r\bmod q\)
satisfy \(\mathbb P(Y=r)\ge 1/(2q)\). Then
\[
        \Ent(Y)\le \log q+\log\gamma_\delta,
\]
where
\[
        \gamma_\delta:=
        \exp\left(
        \frac{\delta}{2}\log 2+
        \left(1-\frac{\delta}{2}\right)
        \log\frac{1-\delta}{1-\delta/2}
        \right)<1.
\]
\end{lemma}

\begin{proof}
Let \(T=\{r\bmod q:\mathbb P(Y=r)<1/(2q)\}\). By assumption,
\(|T|\ge \delta q\). Write \(|T|=mq\), so \(m\ge\delta\), and put
\(\rho=\mathbb P(Y\in T)\). Since each point of \(T\) has mass less than
\(1/(2q)\), we have \(\rho\le m/2\).

Decomposing according to whether \(Y\in T\), and using that entropy is maximized
by the uniform distribution on a fixed support, we get
\[
        \Ent(Y)
        \le
        h(\rho)+\rho\log(mq)+(1-\rho)\log((1-m)q),
\]
where \(h(\rho)=-\rho\log\rho-(1-\rho)\log(1-\rho)\). Equivalently,
\(\Ent(Y)\le \log q+F(m,\rho)\), where
\[
        F(m,\rho)
        =
        h(\rho)+\rho\log m+(1-\rho)\log(1-m).
\]

For fixed \(m\),
\[
        \frac{\partial F}{\partial \rho}
        =
        \log\left(\frac{m(1-\rho)}{\rho(1-m)}\right).
\]
Since \(\rho\le m/2<m\), this derivative is positive. Hence the maximum occurs
at \(\rho=m/2\), and so \(F(m,\rho)\le F(m,m/2)\). A direct calculation gives
\[
        F\left(m,\frac m2\right)
        =
        \frac m2\log 2+
        \left(1-\frac m2\right)
        \log\left(\frac{1-m}{1-m/2}\right).
\]
This expression is decreasing for \(m\in(0,1)\). Therefore its maximum under the
constraint \(m\ge\delta\) occurs at \(m=\delta\). Hence
\[
        F(m,\rho)\le F\left(\delta,\frac{\delta}{2}\right)
        =
        \frac{\delta}{2}\log 2+
        \left(1-\frac{\delta}{2}\right)
        \log\frac{1-\delta}{1-\delta/2}
        =
        \log\gamma_\delta.
\]
The lemma follows.
\end{proof}

We now globalize this local entropy saving over products of primes.

\begin{lemma}\label{lem:norm-entropy-alternative}
Let \(A_1,\ldots,A_r\subseteq [M]\). Let \(p_1,\ldots,p_\ell\) be distinct
primes, let \(\Delta=p_1\cdots p_\ell\), and let \(\mathcal Q\) be the set of
products of \(k\) distinct primes among \(p_1,\ldots,p_\ell\). Fix
\(0<\delta<1\). Suppose that there exists some $\alpha>0$ such that for more than $\alpha {\ell\choose k}$ many \(q\in\mathcal Q\), at least one index
\(i\in\{1,\ldots,r\}\) satisfies
\[
        \#\left\{a\bmod q:
        |A_i(a,q)|\ge \frac{|A_i|}{2q}
        \right\}
        \le (1-\delta)q.
\]
Then
\[
        \prod_{i=1}^{r}|A_i|
        \le
        \left(\frac{M}{\Delta}+1\right)^r
        \Delta^r
        \gamma_\delta^{\alpha\ell/k}.
\]
In particular, if \(\Delta=o(M)\), then
\[
        \prod_{i=1}^{r}|A_i|
        \le
        (1+o(1))M^r\gamma_\delta^{\alpha\ell/k}.
\]
\end{lemma}

\begin{proof}
For each \(i=1,\ldots,r\), let \(X_i\) be uniformly distributed on \(A_i\), and
assume that \(X_1,\ldots,X_r\) are independent. Thus
\[
        \Ent(X_1,\ldots,X_r)
        =
        \sum_{i=1}^r \Ent(X_i)
        =
        \log\prod_{i=1}^r |A_i|.
\]

We view reduction modulo \(\Delta\) coordinate-wise via the Chinese remainder
theorem: \(X_i\bmod\Delta\) is identified with
\((X_i\bmod p_1,\ldots,X_i\bmod p_\ell)\). Thus the joint random vector
\((X_1\bmod \Delta,\ldots,X_r\bmod \Delta)\) has \(\ell\) prime coordinates,
where the coordinate corresponding to \(p_j\) is
\((X_1\bmod p_j,\ldots,X_r\bmod p_j)\).

For \(J\subseteq[\ell]\), \(|J|=k\), write \(q_J=\prod_{j\in J}p_j\). The
projection of the above random vector onto the coordinates in \(J\) is exactly
\((X_1\bmod q_J,\ldots,X_r\bmod q_J)\).

Let $\tilde{Q}$ denote the set of all \(q_J\)'s such that there is at least one index \(i\) for which
\[
        \#\left\{a\bmod q_J:
        |A_i(a,q_J)|\ge \frac{|A_i|}{2q_J}
        \right\}
        \le (1-\delta)q_J.
\]
Since \(X_i\) is uniform on \(A_i\), we have
\(\mathbb P(X_i\equiv a\bmod q_J)=|A_i(a,q_J)|/|A_i|\). Therefore at most
\((1-\delta)q_J\) residue classes \(a\bmod q_J\) have probability at least
\(1/(2q_J)\). By Lemma~\ref{lem:local-entropy-delta},
\[
        \Ent(X_i\bmod q_J)
        \le
        \log q_J+\log\gamma_\delta.
\]
For all other indices \(h\ne i\), we use the trivial bound
\(\Ent(X_h\bmod q_J)\le \log q_J\). Since \(X_1,\ldots,X_r\) are independent, their
reductions modulo \(q_J\) are also independent. Hence
$\Ent(X_1\bmod q_J,\ldots,X_r\bmod q_J)
        \le
        r\log q_J+\log\gamma_\delta.$
For those $q_J\notin\tilde{Q}$, we use the trivial bound 
$\Ent(X_1\bmod q_J,\ldots,X_r\bmod q_J)
        \le
        r\log q_J.$

We now apply Shearer's inequality to the random vector
\((X_1\bmod \Delta,\ldots,X_r\bmod \Delta)\), viewed as a vector indexed by the
\(\ell\) prime coordinates \(p_1,\ldots,p_\ell\). Taking \(\mathcal F\) to be the
family of all \(k\)-element subsets of \([\ell]\), each coordinate belongs to
\(\binom{\ell-1}{k-1}\) members of \(\mathcal F\). Therefore
\[
\begin{aligned}
        \Ent(X_1\bmod \Delta,\ldots,X_r\bmod \Delta)
        &\le
        \frac{1}{\binom{\ell-1}{k-1}}
        \sum_{|J|=k}
        \Ent(X_1\bmod q_J,\ldots,X_r\bmod q_J)  \\
        &\le
        \frac{1}{\binom{\ell-1}{k-1}}
        \left(\sum_{|J|=k}
        r\log q_J+\sum_{q_J\in\tilde{Q}}\log\gamma_\delta\right).
\end{aligned}
\]
Now \(\sum_{|J|=k}\log q_J=\binom{\ell-1}{k-1}\log\Delta\), because each prime
\(p_j\) appears in exactly \(\binom{\ell-1}{k-1}\) of the products \(q_J\). Also, by hypothesis, $|\tilde{Q}|\geq \alpha{\ell\choose k}$. Thus
\[
        \Ent(X_1\bmod \Delta,\ldots,X_r\bmod \Delta)
        \le
        r\log\Delta+\frac{\alpha\ell}{k}\log\gamma_\delta.
\]

It remains to pass from residue classes modulo \(\Delta\) back to integers in
\([M]\). Each residue class modulo \(\Delta\) contains at most \(M/\Delta+1\)
integers from \([M]\). Hence, after the residue classes
\(X_1\bmod\Delta,\ldots,X_r\bmod\Delta\) are specified, there are at most
\((M/\Delta+1)^r\) possible choices for \((X_1,\ldots,X_r)\). Therefore
\[
        \Ent(X_1,\ldots,X_r)
        \le
        \Ent(X_1\bmod\Delta,\ldots,X_r\bmod\Delta)
        +r\log\left(\frac{M}{\Delta}+1\right).
\]
Combining this with the previous entropy bound gives
\[
        \log\prod_{i=1}^r |A_i|
        \le
        r\log\Delta
        +
        \frac{\alpha\ell}{k}\log\gamma_\delta
        +
        r\log\left(\frac{M}{\Delta}+1\right),
\]
as required.
\end{proof}

\subsection{Proof of a weaker version of Theorem~\ref{thm:almost-sidon-cubes}} In this subsection, we illustrate how to apply the ``spread-versus-entropy" dichotomy to prove the following weaker version of Theorem~\ref{thm:almost-sidon-cubes}.
\begin{proposition}\label{thm:almost-sidon-cubesweak}
There is an absolute constant \(c>0\) such that the following holds.
Let \(A\subseteq \mathcal{C}_N\) be a $B_3[g]$ set with
\[
        \log g\le c\frac{(\log N)^{1/4}}{(\log\log N)^{1/2}}.
\]
Then we have
\[
        |A|
        \ll
        N
        \exp\left(
        -c\frac{(\log N)^{1/4}}{(\log\log N)^{1/2}}
        \right).
\]
\end{proposition}

Let 
\[E: X^3+Y^3+Z^3 = 0\]
be the Fermat cubic over $\Q$. For any prime $p\neq 3$, this is a smooth elliptic curve over $\F_p$. Write 
\[|E(\F_p)| = p+1-a_p,\]
where $a_p$ is the trace of Frobenius of $E$ modulo $p$. 
Note that the curve $E$ has complex multiplication by $\Z[\omega]$, where $\omega = e^{2\pi i/3}$. Now it follows from classical results of Hecke and Deuring (see, for example, \cite[Proposition 2.16]{S19}) that for a positive proportion of primes $p\equiv 1\pmod 3$, we have $-a_p\geq 0.1\sqrt{p}$. This immediately implies the following lemma, which allows us to estimate the number of affine solutions.

\begin{lemma}\label{lem:local3}
    There exist absolute constants $\theta,\eta,Y_0>0$ such that for any $Y\geq Y_0$, the interval $[Y,2Y]$ contains at least $\theta Y/\log Y$ many primes $p$ satisfying
    \[|\tilde{E}(\F_p)|\geq p^2+\eta p^{3/2},\]
    where 
    \[\tilde{E}(\F_p) = \{(x,y,z)\in(\F_p^\times)^3: x^3+y^3+z^3\equiv 0\pmod p\}.\]
\end{lemma}
Note that the condition $(x,y,z)\in(\F_p^\times)^3$ will only exclude $O(p)$ many solutions from the collection of all affine solutions. 

\begin{proof}[Proof of Proposition~\ref{thm:almost-sidon-cubesweak}]
    Let $B\subseteq [N]$ be such that $A = \{b^3: b\in B\}$. Fix a sufficiently small constant \(0<\delta<1/10\). Let
\(\gamma_{\delta}<1\) be defined as in Lemma~\ref{lem:local-entropy-delta} and \(\eta_\delta:=-\log\gamma_\delta>0\). Set \(Y=0.1\log N\). Let $\mathcal{P}$ be the set of good primes in $[Y/2,Y]$ coming from Lemma~\ref{lem:local3} and define
\[
        \ell:=|\mathcal P|,
        \qquad
        \Delta:=\prod_{p\in\mathcal P}p.
\]
By Lemma~\ref{lem:local3} and the prime number theorem,
\[
        \theta\frac{Y}{\log Y}\leq \ell\leq \frac{2Y}{\log Y},
        \qquad
        \log\Delta\leq \left(\frac{1}{2}+o(1)\right)\log N.
\]
In particular, \(\Delta=o(N)\). Let \(1\le k\le \ell\) to be chosen later, and let \(\mathcal Q\) be the set
of products of \(k\) distinct primes from \(\mathcal P\). For \(q\in\mathcal Q\), define the popular classes
\[
        \mathcal R(q)
        :=
        \left\{r\bmod q:
        |B(r,q)|\ge \frac{|B|}{2q}\right\}.
\]
We say that \(q\in\mathcal Q\) is \emph{popular} if
\(|\mathcal R(q)|>(1-\delta)q\).

We divide into two cases.

\medskip

\noindent\textbf{Case 1: There exists a popular modulus.}

Let \(q\in\mathcal Q\) be popular. By Lemma~\ref{lem:local3} and Chinese remainder theorem,
\[
        \#\left\{(x_1,x_2,x_3)\in ((\Z/q\Z)^\times)^3:
        x_1^3+x_2^3+x_3^3\equiv0\pmod q\right\}
        \geq \prod_{p\mid q}(p^2+\eta p^{3/2}).
\]
We can partition $((\Z/q\Z)^\times)^3$ into equivalence classes: define $(x_1,x_2,x_3)\sim (y_1,y_2,y_3)$ if and only if there exists some $t\in (\Z/q\Z)^\times$ such that $(x_1,x_2,x_3) = (ty_1,ty_2,ty_3)$. Then each equivalence class contains $\varphi(q)$ many elements. By definition,
\[\varphi(q) = q\prod_{p\mid q}\left(1-\frac{1}{p}\right)\geq q\left(1-\sum_{p\mid q}\frac{1}{p}\right)\geq q\left(1-\frac{2k}{Y}\right).\]
If we choose $k<Y/100$, then $\varphi(q)\geq (1-1/50)q$.

On the other hand, given any solution $(x_1,x_2,x_3)\in ((\Z/q\Z)^\times)^3$ to $E$ modulo $q$, every element in its equivalence class $[x_1:x_2:x_3]$ is also a solution. Moreover, the number of $t\in (\Z/q\Z)^\times$ such that $tx_i\notin \mathcal{R}(q)$ for some $1\leq i\leq 3$ is at most $3\delta q$. Therefore by choosing $\delta$ to be sufficiently small, the equivalence class $[x_1:x_2:x_3]$ will contain at least $q/2$ many elements with each coordinate belonging to $\mathcal{R}(q)$. It follows that 
\[\begin{aligned}
        &\#\{(x_1,x_2,x_3)\in ((\Z/q\Z)^\times)^3:
        x_1^3+x_2^3+x_3^3\equiv0\pmod q,\ x_i\in\mathcal{R}(q)\text{ for all }i\}\\
        &\qquad\geq \frac{q}{2\varphi(q)}\prod_{p\mid q}(p^2+\eta p^{3/2})\gg \prod_{p\mid q}(p^2+\eta p^{3/2}).
\end{aligned}\]

For each such residue vector \((x_1,x_2,x_3)\), we have
\(|B(x_i,q)|\ge |B|/(2q)\) for every \(i\). Therefore
\[
\begin{aligned}
\#\{(b_1,b_2,b_3)\in B^3:
        q\mid b_1^3+b_2^3+b_3^3\} \gg
        \prod_{p\mid q}(p^2+\eta p^{3/2})\left(\frac{|B|}{2q}\right)^3
        \gg
        \frac{|B|^3}{q}\exp\left(c'\frac{k}{\sqrt{Y}}\right)
\end{aligned}
\]
for some absolute constant $c'>0$. 

On the other hand, since \(|b_1^3+b_2^3+b_3^3|\le 3N^3\) for \(b_i\in[N]\) and \(q\le\Delta=o(N)\), the number of integers \(m\) with \(|m|\le 3N^3\) and
\(q\mid m\) is \(O(N^3/q)\). Using
\(R_{A,3}(m)\le g\) for every \(m\), we get
\[
        \#\{(b_1,b_2,b_3)\in B^3:q\mid b_1^3+b_2^3+b_3^3\}
        \ll g\frac{N^3}{q}.
\]
Combining the lower and upper bounds yields
\[
    |B|^3\ll g N^3 \exp\left(-\frac{c'k}{\sqrt{Y}}\right).
\]

\medskip

\noindent\textbf{Case 2: No popular modulus exists.}

Then for every \(q\in\mathcal Q\), we have
\(|\mathcal R(q)|\le (1-\delta)q\). By
Lemma~\ref{lem:norm-entropy-alternative}, we
obtain
\[
        |B|^3
        \le
        (1+o(1))N^3\gamma_\delta^{\ell/k}
        =
        (1+o(1))N^3\exp\left(-\eta_\delta\frac{\ell}{k}\right).
\]

\medskip

Combining the two cases, for every \(1\le k\le\ell\),
\[
        |B|^3
        \ll
        N^3
        \max\left\{
        g \exp\left(-\frac{c'k}{\sqrt{Y}}\right),
        \exp\left(-\eta_\delta\frac{\ell}{k}\right)
        \right\}.
\]

Choose \(k=\lfloor (\eta_\delta\ell/c')^{1/2}Y^{1/4}\rfloor\). When $Y$ is sufficiently large, this satisfies \(k\le\ell\) and \(k<Y/100\). Also,
\[
        \frac{c'k}{\sqrt{Y}}=(1+o(1))(c'\eta_\delta\ell)^{1/2}/Y^{1/4},
        \qquad
        \eta_\delta\frac{\ell}{k}
        =(1+o(1))(c'\eta_\delta\ell)^{1/2}/Y^{1/4}.
\]
Taking \(c>0\) in the hypothesis sufficiently small, the assumption
\(\log g\le c(\log N)^{1/4}/(\log\log N)^{1/2}\) implies
\[
        \log g\le \frac12(c'\eta_\delta\ell)^{1/2}/Y^{1/4}
\]
for all sufficiently large \(N\). Hence \(|B|^3\ll N^3\exp(-c''\ell^{1/2}/Y^{1/4})\). Since
\(\ell\gg \log N/\log\log N\), this gives
\[
        |A|^3 = |B|^3
        \le
        CN^3
        \exp\left(
        -c\frac{(\log N)^{1/4}}{(\log\log N)^{1/2}}
        \right),
\]
after adjusting the constants \(c,C>0\). This proves the theorem.
\end{proof}

We end with some remarks on the framework developed in this section.

\begin{remark}\label{remark:Shearer}
    The proof technique used here can be applied to any polynomial representations as long as we have a local arithmetic input (more precisely, a large enough local gain for many primes) as in Lemma~\ref{lem:local3}. 
    
    In particular, for $B_2[g]$ sets in squares, we have $2p-1$ solutions to $x^2+y^2=0$ in $\F_p$ for primes $p\equiv1\pmod 4$. For $B_4[g]$ sets in fourth powers, by using basic properties of Gauss sums, we have $p^3+p^2-p$ solutions to $x_1^4+x_2^4+x_3^4+x_4^4=0$ in $\F_p$ for primes $p\equiv 3\pmod 4$. For a norm form $F$ associated to a number field of degree $r$, by Chebotarev density theorem, we have $(r-o(1))p^{r-1}$ many solutions to $F(x_1,\ldots,x_r)=0$ in $\F_p$ for a positive density subset of primes. More precise versions of these arithmetic inputs will be discussed in Section~\ref{sec:weight}. 
    
    However, in this framework, the final saving is obtained by balancing the congruence gain against the
entropy loss, which results in a ``square-root loss" compared to the ideal spread setting. For example, for a $B_2[g]$ set $A$ in $\mathcal{S}_N$, we get \[
        |A|
        \ll_{g}
        N
        \exp\left(
        -c\sqrt{\frac{\log N}{\log\log N}}
        \right),
\]
which is weaker than Theorem~\ref{thm:almost-sidon-squares}. Similarly, this approach only yields weaker bounds for norm forms and $B_4[g]$ sets in fourth powers.

In Section~\ref{sec:weight}, we develop a weighted version of this framework to bypass this ``square-root loss", by introducing a suitable weight function. However, we believe this framework is more flexible (although sometimes produces a weaker bound) since it is often tricky to find an efficient weight function.
\end{remark}

\begin{remark}\label{rem:higherh}
The proof techniques used in Sections~\ref{sec:entropy-sieve} and~\ref{sec:weight} do not extend to $B_h[g]$ sets in perfect $k$-th powers with $h\geq 5$. The main reason is that for the diagonal equation $\sum_{i=1}^h x_i^k=0$ over $\F_p$, using basic facts about Jacobi sums, the number of affine solutions is $p^{h-1}+O(p^{h/2})$ (see for example \cite[Theorem 6.36]{LN}). Since $\sum_{p} p^{-3/2}$ converges, the saving coming from the fluctuation argument for the error term is at most a constant factor for $h\geq 5$. 
\end{remark}

\section{A weighted entropy-enhanced sieve and its applications}\label{sec:weight}
The preceding section used a spread-versus-entropy dichotomy. In that
approach we first choose $\ell$ ``good" primes and then choose a product
\[
q=p_1\cdots p_k
\]
consisting of \(k\) good primes where we have local gain for each $p_i$. If the underlying set is sufficiently spread out modulo
\(q\), then the local splitting gives many solutions to the relevant congruence.
If this fails for every such \(q\), Shearer's inequality gives an entropy loss.
To balance the congruence gain against the
entropy loss, we can only choose $k$ to be much smaller than $\ell$, necessarily paying a cost of ``square-root loss".

In this section, we develop a weighted, entropy-enhanced sieve that bypasses this dichotomy, provided that there is an efficient corresponding weight function. Roughly speaking, instead
of selecting one modulus \(q\) as above, we average over all squarefree products of the
good primes simultaneously. 

Let $A$ be a set with bounded algebraic multiplicity. Let $X_1, X_2,\ldots, X_r$ be some auxiliary sets associated to $A$. Let \(\mathcal P\) be a set of $\ell$ good primes. For each \(p\in\mathcal P\) and $(x_1,x_2,\ldots, x_r)\in X_1\times X_2 \times \cdots \times X_r$, define a weight function $m_p(x_1,x_2,\ldots, x_r)$ which represents the local contribution of the tuple $(x_1,x_2,\ldots, x_r)$ modulo $p$.
We then introduce the weight
\[
K(x_1,x_2,\ldots, x_r)
:=
\prod_{p\in\mathcal P}
\left(1+m_p(x_1,x_2,\ldots, x_r)\right)
\]
representing the global contribution of $(x_1,x_2,\ldots, x_r)$. Indeed, expanding the product shows that \(K\) is a positive weighted average of the contributions from squarefree products of good primes.

We then consider the following sum:
\begin{equation}\label{eq:S}
\mathcal{S}=\sum_{x_1\in X_1}\sum_{x_2\in X_2}\cdots\sum_{x_r\in X_r}K(x_1,x_2,\ldots, x_r).    
\end{equation}
For suitable local weight functions $m_p$, we can upper-bound $\mathcal{S}$ by a bounded multiple of a partial sum of the corresponding multiplicative function. As for the lower bound, we use
an entropy-defect estimate based on  Proposition~\ref{prop:weightentropy} below. Comparing the lower bound and upper bound on $S$ then yields an upper bound on $
|A|$. This replaces the Shearer dichotomy (where we only have the saving from $k$ primes for $k$ much smaller than $\ell$), and
keeps the full \(\ell\)-scale saving. This recovers the saving one would obtain in the fully spread case.

The weighted argument above is flexible, but the applications in this paper
only require the following convenient formulation.  The theorem packages the
ingredients that will be verified in each case: a bounded representation
hypothesis, a set of good primes, and local weights with uniform marginals and
a gain on the zero set of the relevant polynomial.

\begin{theorem}\label{thm:weight}
Let $r\geq 2$. Suppose \(F(x_1,\ldots,x_r)\) is an integer-valued polynomial of degree at most \(r\) defined over rationals, and
\(B_1, B_2,\ldots, B_r \subset [N]\). Suppose
\[
R_F(n)
:= \#\bigl\{(b_1,\ldots,b_r)\in B_1\times\cdots\times B_r :
F(b_1,\ldots,b_r)=n\bigr\}
\le g
\]
for every \(n\). Let $s\geq 2$ be a positive integer and $t\geq 1$ be a real number. Let $\mathcal{P}$ be a set of $\ell$ primes. Put $       \Delta:=\prod_{p\in\mathcal P}p,$ and assume that \(\Delta\le N\) and $\Delta^s\leq N^r$. Suppose that for each \(p\in \mathcal{P}\), there is a local weight function \[m_p:\mathbb{F}_p^r \to \mathbb{R}_{\ge 0}\]
satisfying the following two conditions:

(1) For every \(1\le j\le r\) and every fixed \(x_j\in \F_p\), we have
\[
\sum_{(x_i)_{i\ne j}} m_p(x_1,\ldots,x_r)=p^{r-1}.
\]

(2) There exists \(\delta_p\in(0,1)\) such that for every
\(\mathbf{b}=(b_1,\ldots,b_r)\in\mathbb Z^r\),
\[
        m_p(\overline{\mathbf{b}})
        \le
        p(1-\delta_p)\mathbf 1_{p\mid F(\mathbf{b})}
        +
        p^t\mathbf 1_{p^s\mid F(\mathbf{b})},
\]
where \(\overline{\mathbf{b}} \in \F_{p^r}\) denotes the reduction of \(\mathbf{b}\) modulo \(p\).

Then
\[
\prod_{i=1}^{r} |B_i|\ll_{F} g^{1/r}N^r
\exp\left(-\frac{1}{2r}\sum_{p\in \mathcal{P}}(\delta_p-p^{t-s})\right)+2^{\ell/r}g^{1/r}N^{r-1} \Delta^{t/r},
\]
where the implied constant only depends on $F$.
\end{theorem}
In each application below, we choose a set of primes \(\mathcal P\) and local weights \(m_p\) satisfying the hypotheses of Theorem~\ref{thm:weight}. The parameters \(s,t,\delta_p\) are determined by the corresponding local estimate; the primes are chosen so that the first term gives the required saving, while \(\Delta=\prod_{p\in\mathcal P}p\) is small enough for the second term to be absorbed into the final bound.

\subsection{Lower bound estimate}
For suitable local weight functions, we lower-bound the quantity
\(\mathcal S\) defined in~\eqref{eq:S} using a simple product-space resampling model.
The key estimate is the following half-resampling lemma.  This resampling
argument is reminiscent of reverse hypercontractivity on product spaces; see
Mossel, Oleszkiewicz, and Sen~\cite{MOS}.  We shall only need the elementary
self-contained form proved below.

\begin{lemma}\label{lem:resampling}
Let $r,\ell$ be positive integers with $r\geq 2$. For each $1\leq i\leq r$ and $1\leq k\leq \ell$, let $\Omega_{i,k}$ be a nonempty finite set. For each $1\leq i\leq r$, let $        \Omega_i=\Omega_{i,1}\times\cdots\times\Omega_{i,\ell}.$ For each \(1\leq k\leq \ell\), let
$X_k=(X_{1,k},\ldots,X_{r,k})$ be a random variable on $      \Omega_{1,k}\times\cdots\times\Omega_{r,k}.$
Assume that \(X_1,\ldots,X_\ell\) are independent, and each $X_{i,k}$ is uniformly distributed on $\Omega_{i,k}$ for each $1\leq i\leq r$ and $1\leq k \leq \ell$. For each \(1\leq k \leq \ell\), independently, define
\[
        X_k^*=(X_{1,k}^*,\ldots,X_{r,k}^*)
\]
as follows: with probability \(1/2\), keep \(X_k\); with probability \(1/2\), replace it by an independent uniformly distributed element of $\Omega_{1,k}\times\cdots\times\Omega_{r,k}.$ For each $1\leq i \leq r$, let
\[
        W_i=(X_{i,1}^*,\ldots,X_{i,\ell}^*)
\]
be the random variable on $\Omega_i$.
Then, for every choice of nonnegative functions $       U_i:\Omega_i\to \mathbb R_{\ge 0}$, we have
\[
        \mathbb E \prod_{i=1}^r U_i(W_i)
        \ge
        \prod_{i=1}^r
        \left(\mathbb E_{\Omega_i} U_i^{1/r}\right)^r.
\]    
\end{lemma}

\begin{proof}
We prove the lemma by induction on \(\ell\).

First assume \(\ell=1\). Write
\[
        X=(X_1,\ldots,X_r)
        \in \Omega_1\times\cdots\times\Omega_r,
\]
where each \(X_i\) is uniformly distributed on \(\Omega_i\). Let \(X^*\) be
obtained from \(X\) by keeping \(X\) with probability \(1/2\), and replacing it
by an independent product-uniform element of
\(\Omega_1\times\cdots\times\Omega_r\) with probability \(1/2\). Then
\[
\begin{aligned}
        \mathbb E\prod_{i=1}^r U_i(X_i^*)
        &=
        \frac12 \mathbb E\prod_{i=1}^r U_i(X_i)
        +
        \frac12 \prod_{i=1}^r \mathbb E_{\Omega_i}U_i.
\end{aligned}
\]
Apply Corollary~\ref{cor:r} with $  A_i=U_i(X_i)$ for $1\leq i \leq r$, we have
\[
        \mathbb E\prod_{i=1}^r A_i+
        \prod_{i=1}^r\mathbb E A_i
        \ge
        2\prod_{i=1}^r
        \left(\mathbb E A_i^{1/r}\right)^r .
\]
For each $1\leq i \leq r$, since \(X_i\) is uniform on \(\Omega_i\), we have $\mathbb E A_i=\mathbb E_{\Omega_i}U_i$ and $\mathbb E A_i^{1/r}=\mathbb E_{\Omega_i}U_i^{1/r}.$ Therefore
\[
        \mathbb E\prod_{i=1}^r U_i(X_i^*)
        \ge
        \prod_{i=1}^r
        \left(\mathbb E_{\Omega_i}U_i^{1/r}\right)^r.
\]
This proves the case \(\ell=1\).

Now assume \(\ell\ge2\), and suppose the result is known for \(\ell-1\). For each $1\leq i \leq r$, write
\[
        \Omega_i'
        =
        \Omega_{i,1}\times\cdots\times\Omega_{i,\ell-1},
        \qquad
        W_i'=(X_{i,1}^*,\ldots,X_{i,\ell-1}^*)\in\Omega_i', \qquad 
        W_i=(W_i',X_{i,\ell}^*).
\]
Condition on \(W_1',\ldots,W_r'\). Since the last coordinate is independent of
the first \(\ell-1\) coordinates, the conditional law of $       (X_{1,\ell}^*,\ldots,X_{r,\ell}^*)$ is exactly the one-coordinate half-resampled law associated to $X_\ell=(X_{1,\ell},\ldots,X_{r,\ell}).$
For fixed \(w_i'\in\Omega_i'\), apply the case \(\ell=1\) to the functions $t\mapsto U_i(w_i',t)$, $t\in\Omega_{i,\ell}.$ This gives
\begin{equation}\label{eq:conditional-resampling}
        \mathbb E\left[
        \prod_{i=1}^r U_i(W_i',X_{i,\ell}^*)
        \,\middle|\, W_1',\ldots,W_r'
        \right]
        \ge
        \prod_{i=1}^r A_i(W_i')^r,
\end{equation}
where
\[
        A_i(w_i')
        =
        \mathbb E_{t\in\Omega_{i,\ell}}
        U_i(w_i',t)^{1/r}.
\]
Taking expectations on both sides of inequality~\eqref{eq:conditional-resampling}, we obtain
\[
        \mathbb E\prod_{i=1}^r U_i(W_i)
        \ge
        \mathbb E\prod_{i=1}^r A_i(W_i')^r.
\]

We now apply the induction hypothesis on the product spaces
\(\Omega_1',\ldots,\Omega_r'\), with the functions $       V_i(w_i'):=A_i(w_i')^r$ for each $1\leq i \leq r$. It gives
\[
        \mathbb E\prod_{i=1}^r A_i(W_i')^r
        =
        \mathbb E\prod_{i=1}^r V_i(W_i')
        \ge
        \prod_{i=1}^r
        \left(\mathbb E_{\Omega_i'} V_i^{1/r}\right)^r.
\]
Since \(V_i^{1/r}=A_i\), we have 
\[
        \mathbb E_{\Omega_i'} V_i^{1/r}
        =
        \mathbb E_{\Omega_i'} A_i
        =
        \mathbb E_{\Omega_i} U_i^{1/r}.
\]
Combining the last three displayed inequalities yields
\[
        \mathbb E\prod_{i=1}^r U_i(W_i)
        \ge
        \prod_{i=1}^r
        \left(\mathbb E_{\Omega_i}U_i^{1/r}\right)^r.
\]
This completes the induction and proves the lemma.
\end{proof}

The following proposition will be used to give a lower bound on $\mathcal{S}$ based on entropy defects.

\begin{proposition}\label{prop:weightentropy}
Let \(r\ge 2\), and let \(\mathcal{P}\) be a  set with \(|\mathcal{P}|=\ell\).
For each \(p\in \mathcal{P}\), let $I_{1,p},\ldots,I_{r,p}$ be finite nonempty sets, and let the local weight function \[        m_p:I_{1,p}\times\cdots\times I_{r,p}\to \mathbb R_{\ge 0}\] satisfy the condition
\begin{equation}\label{eq:constantsum}
        \sum_{(x_i)_{i\ne j}} m_p(x_1,\ldots,x_r)
        =
        \prod_{i\ne j}|I_{i,p}|
\end{equation}
for every \(1\le j\le r\) and every fixed \(x_j\in I_{j,p}\). For each $1\le j\le r$, set
\[
        I_j:=\prod_{p\in \mathcal{P}} I_{j,p}.
\]
Define 
\[
        K(x_1,\ldots,x_r)
        :=
        \prod_{p\in \mathcal{P}}
        \bigl(1+m_p(x_{1,p},\ldots,x_{r,p})\bigr).
\]
If \(Z_1,\ldots,Z_r\) are independent random variables with \(Z_j\) taking
values in \(I_j\), then
\[
        \mathbb E K(Z_1,\ldots,Z_r)
        \ge
        2^\ell
        \exp\left(
        -(r-1)\sum_{j=1}^r
        \bigl(\log |I_j|-\Ent(Z_j)\bigr)
        \right).
\]
\end{proposition}
\begin{proof}
    For any $1\leq j\leq r$, define 
    \[f_j(x) := |I_j|\mathbb P(Z_j = x),\qquad \forall x\in I_j.\]
    For each $p\in\mathcal{P}$, let $X_p$ be a random variable defined on $I_{1,p}\times\cdots\times I_{r,p}$ such that 
    \begin{equation}\label{eq:prob_Xp}
    \mathbb P(X_p = (x_{1,p},\ldots,x_{r,p})) = \frac{m_p(x_{1,p},\ldots,x_{r,p})}{\prod_{j=1}^r|I_{j,p}|}
    \end{equation}
    for all $(x_{1,p},\ldots,x_{r,p})\in I_{1,p}\times\cdots\times I_{r,p}$. Note that this random variable is well-defined since equation~\eqref{eq:constantsum} guarantees that 
    \[\sum_{x_{1,p}\in I_{1,p}}\cdots\sum_{x_{r,p}\in I_{r,p}}\mathbb P(X_p = (x_{1,p},\ldots,x_{r,p})) = 1.\]
    Also, for any $1\leq j\leq r$ and any fixed $y\in I_{j,p}$,  equation~\eqref{eq:constantsum} implies that
    \[\mathbb P(X_{j,p} = y) = \sum_{(x_{i,p})_{i\neq j}}\mathbb P(X_p = (x_{1,p},\ldots,x_{j-1,p},y,x_{j+1,p},\ldots,x_{r,p})) = \frac{1}{|I_{j,p}|},\]
    that is, $X_{j,p}$ is uniformly distributed on $I_{j,p}$. 
    We further assume that $(X_p)_{p\in\mathcal{P}}$ is independent. 
    
    For each $p\in\mathcal{P}$, let $X_p^*=(X_{1,p}^*,\ldots, X_{r,p}^*)$ be obtained from $X_p$ by the half-resampling rule of Lemma~\ref{lem:resampling}. For each $1\leq i\leq r$, let $W_i=(X_{i,p}^*)_{p\in \mathcal{P}}$ be defined as in Lemma~\ref{lem:resampling}. Next, we prove the following key observation.
    \begin{claim}
We have
\begin{equation}\label{eq:product_form}
2^{-\ell}\mathbb E K(Z_1,\ldots,Z_r) = \mathbb E\left(\prod_{j=1}^rf_j(W_j)\right).    
\end{equation}
\end{claim}
\begin{poc}
For each \(p\in\mathcal P\), write
\[
        J_p:=I_{1,p}\times\cdots\times I_{r,p},
        \qquad
        M_p:=|J_p|=\prod_{j=1}^r |I_{j,p}|.
\]
Recall that \(I_j=\prod_{p\in\mathcal P}I_{j,p}\).  If
\(a=(a_1,\ldots,a_r)\in I_1\times\cdots\times I_r\), with
\(a_j=(a_{j,p})_{p\in\mathcal P}\), write
\[
        a^{(p)}:=(a_{1,p},\ldots,a_{r,p})\in J_p.
\]
We use the same notation \(x^{(p)}\) for another tuple
\(x=(x_1,\ldots,x_r)\).

By the half-resampling rule, for each $p\in \mathcal{P}$ and \(u,v\in J_p\),
\[
        \mathbb P(X_p^*=v\mid X_p=u)
        =
        \frac12\mathbf 1_{v=u}+\frac{1}{2M_p}.
\]
Consequently, for any
\[
        a=(a_1,\ldots,a_r),\qquad x=(x_1,\ldots,x_r)
        \in I_1\times\cdots\times I_r,
\]
where \(a_j=(a_{j,p})_{p\in\mathcal P}\) and
\(x_j=(x_{j,p})_{p\in\mathcal P}\), independence over the primes gives
\begin{equation}\label{eq:cond_prob_W}
\mathbb P((W_1,\ldots,W_r)=a\mid X_p=x^{(p)}\text{ for all }p)
=
\prod_{p\in\mathcal P}
\left(
        \frac12\mathbf 1_{a^{(p)}=x^{(p)}}
        +
        \frac{1}{2M_p}
\right).
\end{equation}

Using equations~\eqref{eq:prob_Xp} and \eqref{eq:cond_prob_W}, as \(a\) and \(x\) range over \(I_1\times\cdots\times I_r\), we have
\begin{align*}
\mathbb E\left(\prod_{j=1}^r f_j(W_j)\right)
&=
\sum_{a}\sum_{x}
        \prod_{p\in\mathcal P}
        \left(
        \frac12\mathbf 1_{a^{(p)}=x^{(p)}}
        +
        \frac{1}{2M_p}
        \right)
        m_p(x^{(p)})
        \prod_{j=1}^r\frac{f_j(a_j)}{|I_j|}                         \\
&=
\sum_{a}
        \prod_{p\in\mathcal P}
        \left[
        \frac12 m_p(a^{(p)})
        +
        \frac{1}{2M_p}\sum_{u\in J_p}m_p(u)
        \right]
        \prod_{j=1}^r\frac{f_j(a_j)}{|I_j|}.
\end{align*}
Note that equation~\eqref{eq:constantsum} implies that $\sum_{u\in J_p}m_p(u)=M_p$ for each $p\in \mathcal{P}$. Therefore,
\begin{align*}
\mathbb E\left(\prod_{j=1}^r f_j(W_j)\right)
&=
2^{-\ell}
\sum_{a\in I_1\times\cdots\times I_r}
        \prod_{p\in\mathcal P}\bigl(1+m_p(a^{(p)})\bigr)
        \prod_{j=1}^r\frac{f_j(a_j)}{|I_j|}                         \\
&=
2^{-\ell}
\sum_{a\in I_1\times\cdots\times I_r}
        K(a_1,\ldots,a_r)
        \prod_{j=1}^r\mathbb P(Z_j=a_j)                             \\
&=
2^{-\ell}\mathbb E K(Z_1,\ldots,Z_r),
\end{align*}
as claimed.
\end{poc}
    We now lower-bound the expectation on the right-hand side of equation~\eqref{eq:product_form}. Apply Lemma~\ref{lem:resampling} with $U_i = f_i$, we get
\begin{equation}\label{eq:resample_lower_bd}
    \mathbb E\left(\prod_{j=1}^rf_j(W_j)\right)\geq \prod_{j=1}^r\left(\mathbb E_{I_j}f_i^{1/r}\right)^r.
    \end{equation}
Next we estimate \(\mathbb E_{I_j}f_j^{1/r}\) for each $1\leq j\leq r$. Note that
\[
\begin{aligned}
        \mathbb E_{I_j}f_j^{1/r}
        =
        \frac1{|I_j|}\sum_{x\in I_j}f_j(x)^{1/r}  
        =
        \sum_{x:f_j(x)>0}\mathbb P(Z_j=x)f_j(x)^{-(r-1)/r}  
        =
        \mathbb E\exp\left(-\frac{r-1}{r}\log f_j(Z_j)\right).
\end{aligned}
\]
Since the exponential function is convex, by Jensen's inequality,
\[
        \mathbb E_{I_j}f_j^{1/r}=
        \mathbb E\exp\left(-\frac{r-1}{r}\log f_j(Z_j)\right)
        \ge
        \exp\left(-\frac{r-1}{r}\mathbb E\log f_j(Z_j)\right).
\]
But
\[
        \mathbb E\log f_j(Z_j)
        =
        \sum_{x\in I_j}\mathbb P(Z_j=x)
        \log\bigl(|I_j|\mathbb P(Z_j=x)\bigr)
        =
        \log |I_j|-\Ent(Z_j).
\]
Hence
\begin{equation}\label{eq:1/r_f}
\mathbb E_{I_j}f_j^{1/r}
        \ge
        \exp\left(-\frac{r-1}{r}(\log |I_j|-\Ent(Z_j))\right).    
\end{equation}

Combining equation~\eqref{eq:product_form} with inequalities~\eqref{eq:resample_lower_bd} and~\eqref{eq:1/r_f}, we conclude that
\[2^{-\ell}\mathbb EK(Z_1,\ldots,Z_r)\geq \exp\left(-(r-1)\sum_{j=1}^r(\log|I_j|-\Ent(Z_j))\right).\qedhere\]
\end{proof}

\subsection{Proof of Theorem~\ref{thm:weight}}
For each $b_i \in B_i$ with $1\leq i\leq r$, let 
   \[K(b_1,b_2,\ldots, b_r)=\prod_{p \in 
    \mathcal{P}} \big(1+m_p(b_1,b_2,\ldots, b_r)\big)\]
    Set
    \[\mathcal{S}:= \sum_{\substack{b_i \in B_i \\1\leq i\leq r}}K(b_1,b_2,\ldots, b_r).\]

    We first prove a lower bound for $\mathcal{S}$. Let
    \(Z_1,\ldots,Z_r\) be independent random variables such that for $1\leq i \leq r$, \(Z_i\) is uniformly
distributed on \(B_i\). For each $1\leq i \leq r$, let \(\overline Z_i\) denote the reduction of \(Z_i\)
modulo \(\Delta\), viewed as a random variable taking values in
\[
I=\mathbb Z/\Delta\mathbb Z
\simeq
\prod_{p\in\mathcal P}\mathbb F_p.
\]
By Proposition~\ref{prop:weightentropy},
\begin{equation}\label{eq:forth_lower_bd}
\mathcal S
=
\prod_{i=1}^{r} |B_i| \cdot \mathbb E K(\overline{Z_1},\overline{Z_2}, \ldots, \overline{Z_r})
\ge
\prod_{i=1}^{r} |B_i| \cdot 2^\ell
\exp\left(
        -(r-1)\sum_{j=1}^r
        \bigl(\log |I|-\Ent(Z_j)\bigr)
        \right).
\end{equation}

Let $1\leq i \leq r$. We now bound the entropy defect of \(\overline{Z_i}\). Note that \(\mathbb P(\overline{Z_i}=x)=|B_i(x,\Delta)|/|B_i|\) for each $x \in I$. Since $\Delta\leq N$, each residue class modulo \(\Delta\) contains at most $\left\lceil\frac N\Delta\right\rceil
\le
\frac{2N}{\Delta}$
integers from \(B_i \subseteq [N]\), we have \(\Delta\mathbb P(\overline{Z_i}=x)
\le 2N/|B_i|\)  whenever \(\mathbb P(\overline{Z_i}=x)>0\). Therefore
\[
\log\Delta-\Ent(\overline{Z_i})
=
\sum_x
\mathbb P(\overline{Z_i}=x)
\log\bigl(\Delta\mathbb P(\overline{Z_i}=x)\bigr)
\le
\log\frac{2N}{|B_i|}.
\]
Substituting this into inequality~\eqref{eq:forth_lower_bd}, we obtain
\begin{equation}\label{eq:Slb}
\mathcal S
\ge
\prod_{i=1}^{r} |B_i| \cdot 2^\ell
\exp\left(
        -(r-1)\sum_{j=1}^r
        \log\frac{2N}{|B_i|}
        \right) \gg \frac{2^{\ell}(\prod_{i=1}^{r} |B_i|)^r}{N^{r(r-1)}}.
\end{equation}

We next prove an upper bound for $\mathcal{S}$. By the assumption, for each $(x_1,x_2,\ldots, x_r)\in \F_p^r$, we have
\[
m_p(x_1,\ldots,x_r)
\le
p(1-\delta_p)\mathbf{1}_{p \mid F(x_1,\ldots,x_r)}+p^t \mathbf{1}_{p^s \mid F(x_1,\ldots,x_r)}.
\]

Since $F$ is a polynomial of degree at most $r$, there is a constant $C\geq 1$ such that $|F(x_1,x_2,\ldots, x_r)|\leq CN^r$ for all $x_1,x_2,\ldots, x_r\in [N]$. Then, since $R_F(n)\leq g$ for all integer $n$, we have
\[\mathcal{S}=\sum_{\substack{b_i \in B_i \\1\leq i\leq r}}K(b_1,b_2,\ldots, b_r)\leq g\sum_{|n|\leq CN^r}\prod_{p\in\mathcal{P}}\left(1+p(1-\delta_p)\mathbf{1}_{p\mid n}+p^t\mathbf{1}_{p^s\mid n}\right).\]
Define a multiplicative function $u:\N\rightarrow\R$ so that $u(p) = p(1-\delta_p), u(p^s) = p^t$ for $p\in\mathcal{P}$, and $u(p^a) = 0$ for $a\notin \{0,1,s\}$ or $p\notin\mathcal{P}$. Since \(u\) is supported on divisors of \(\Delta^s\), and 
\(\Delta^s\leq N^r\), we have \(u(d)=0\) unless \(d\le CN^r\). Then
\[\sum_{|n|\leq CN^r}\prod_{p\in\mathcal{P}}\left(1+p(1-\delta_p)\mathbf{1}_{p\mid n}+p^t\mathbf{1}_{p^s\mid n}\right) = \sum_{|n|\leq CN^r}\sum_{d\mid n}u(d).\]
Thus, a standard estimate for multiplicative functions gives
\begin{align*}
\sum_{|n|\leq CN^r}\sum_{d\mid n}u(d)&\leq \sum_{d\leq CN^r}u(d)\left(\frac{2CN^r}{d}+1\right)\leq 2CN^r\sum_{d\leq CN^r}\frac{u(d)}{d}+\sum_{d\leq CN^r}u(d) \\
&\leq 2CN^r\prod_{p\in\mathcal{P}}\left(1+1-\delta_p+p^{t-s}\right)+\prod_{p\in\mathcal{P}}\left(1+p(1-\delta_p)+p^t\right)\\
&\ll N^r2^\ell\exp\left(-\frac{1}{2}\sum_{p\in\mathcal{P}}(\delta_p-p^{t-s})\right)+3^\ell \Delta^t.
\end{align*}
Hence
\begin{equation}\label{eq:Sub}
\mathcal{S}\ll gN^r2^\ell\exp\left(-\frac{1}{2}\sum_{p\in\mathcal{P}}(\delta_p-p^{t-s})\right)+g3^\ell \Delta^t.
\end{equation}

Comparing the lower and upper bounds on $\mathcal{S}$ from inequalities~\eqref{eq:Slb} and~\eqref{eq:Sub}, we obtain 
\[
\frac{2^{\ell}(\prod_{i=1}^{r} |B_i|)^r}{N^{r(r-1)}}\ll gN^r2^\ell\exp\left(-\frac{1}{2}\sum_{p\in\mathcal{P}}(\delta_p-p^{t-s})\right)+g3^\ell \Delta^t
\]
which implies 
\[
\prod_{i=1}^{r} |B_i|\ll g^{1/r}N^r
\exp\left(-\frac{1}{2r}\sum_{p\in \mathcal{P}}(\delta_p-p^{t-s})\right)+2^{\ell/r}g^{1/r}N^{r-1}\Delta^{t/r},
\]
as required. \qed

\subsection{Application I: $B_3[g]$ subsets of cubes}
\begin{proof}[Proof of Theorem~\ref{thm:almost-sidon-cubes}]
Let 
\[
E: x^3+y^3+z^3=0
\]
as a (projective) curve  and write
\[
|E(\F_p)|=p+1-a_p.
\]
Note there are 3 points of the form $[0:y:z] \in E(\F_p)$, and thus there are $p+1-a_p-9=p-a_p-8$ points $[x:y:z]\in E(\F_p)$ with $xyz\neq 0$.

Write $A = \{b^3:b\in B\}$, where $B\subseteq [N]$. Set $Y = 0.1\log N$. Let $\mathcal{P}$ be the set of primes $p\in[Y,2Y]$ so that \(p\equiv 1 \pmod 3\) with Fermat-cubic trace satisfying
\[
        -a_p \ge 0.1 \sqrt p,
\]
and define $\ell = |\mathcal{P}|$, $\Delta = \prod_{p\in\mathcal{P}}p$. Following the proof of Proposition~\ref{thm:almost-sidon-cubesweak}, there exists an absolute constant $\theta>0$ such that
\begin{equation}\label{eq:fermat_cubic_pnt}
\ell\geq \frac{\theta Y}{\log Y},\qquad\log \Delta\leq (0.2+o(1))\log N.
\end{equation}
For each $p\in\mathcal{P}$, define the local weight \(m_p:\mathbb F_p^3\to \mathbb R_{\ge 0}\) by
\[
m_p(x,y,z)=
\begin{cases}
w_p, & xyz\ne 0,\quad \text{and} \quad x^3+y^3+z^3=0,\\
u_p, & \text{exactly one of }x,y,z\text{ is }0,\quad \text{and} \quad x^3+y^3+z^3=0,\\
0, & \text{otherwise},
\end{cases}
\]
where
\[
        u_p=\frac{p^2}{3(p-1)}, \qquad 
        w_p=\frac{p^2-6u_p}{p-a_p-8}.
\]
Next, we check the two conditions in Theorem~\ref{thm:weight}. When \(x=0\), there are \(3(p-1)\) nonzero solutions of $y^3+z^3=0$ and thus
\[
\sum_{y,z\in \F_p} m_p(0,y,z)=3(p-1)u_p=p^2.
\]
While for \(x\ne 0\), there are \(6\) solutions with exactly one of \(y,z\)
zero and \(p-a_p-8\) all-nonzero solutions, so
\[
\sum_{y,z\in \F_p} m_p(x,y,z)=6u_p+(p-a_p-8)w_p=p^2.
\]
The same holds symmetrically in
the other coordinates.

Note that $u_p\leq w_p\leq p(1-\frac{\alpha}{\sqrt{p}})$ for some constant $\alpha>0$. Hence
\[m_p(x,y,z)\leq p\left(1-\frac{\alpha}{\sqrt{p}}\right)\mathbf{1}_{p\mid x^3+y^3+z^3}\]
for all $(x,y,z)\in\F_p^3$. It then follows from Theorem~\ref{thm:weight} and estimate~\eqref{eq:fermat_cubic_pnt} that
\[
|B|^3\ll g^{1/3}N^3\exp\left(-\frac{1}{6}\sum_{p\in\mathcal{P}}\frac{\alpha}{\sqrt{p}}\right),
\]
which implies 
\[
|A|=|B|\ll g^{1/9}N\exp\left(-c\frac{(\log N)^{1/2}}{\log\log N}\right)
\]
for some constant $c>0$.
\end{proof}

\subsection{Application II: $B_4[g]$ sets in fourth powers} Throughout this subsection, let $F(x_1,x_2,x_3,x_4) = x_1^4+x_2^4+x_3^4+x_4^4$.
Using basic properties of Gauss sums, we have the following lemma to help us design local weight functions. 

\begin{lemma}\label{lem:local4}
Let \(p\equiv 3\pmod 4\). For each \(a\in\mathbb F_p\), define
\[
T_p(a)=\#\{(x_2,x_3,x_4)\in\mathbb F_p^3:
a^4+x_2^4+x_3^4+x_4^4=0\}.
\]
Then 
\[
T_p(a)=
\begin{cases}
p^2, & a=0,\\
p^2+p, & a\ne0.
\end{cases}
\]
\end{lemma}

\begin{proof}
Since \(p\equiv3\pmod4\), the maps \(x\mapsto x^4\) and \(x\mapsto x^2\)
have the same image and the same fibre sizes on \(\mathbb F_p\). Thus it
suffices to count the corresponding quadratic equations. Let \(\psi\) be the canonical additive character of \(\mathbb F_p\), and $\chi$ be the quadratic character of $\F_p$. For each $s\in \F_p$, consider the quadratic Gauss sum
\[
G(s)=\sum_{x\in\mathbb F_p}\psi(sx^2).
\]
For \(s\ne0\), we have the following basic property (see for example \cite[Section 5.2]{LN}):
\[
G(s)=\chi(s)G(1),\qquad G(1)^2=\chi(-1)p=-p.
\]
For each $t\in \F_p$, define $N_p(t):=\#\{y\in\mathbb F_p^3:
        y_1^2+y_2^2+y_3^2=t\}.$ Orthogonality gives
\[
N_p(t)
=
\frac1p\sum_{s\in\mathbb F_p}
\psi(-st)G(s)^3                                   =
p^2+\frac{G(1)^3}{p}
\sum_{s\ne0}\chi(s)\psi(-st).
\]
If \(t=0\), the last sum is \(0\). If \(t\ne0\), then $       \sum_{s\ne0}\chi(s)\psi(-st)=\chi(-t)G(1).$ Therefore
\[
        N_p(t)=
        \begin{cases}
        p^2, & t=0,\\
        p^2+p\chi(-t), & t\ne0.
        \end{cases}
\]
Taking \(t=-a^4\) gives the required count on $T_p(a)$.
\end{proof}

Next we present the proof of Theorem~\ref{thm:almost-sidon-fourth}.

\begin{proof}[Proof of Theorem~\ref{thm:almost-sidon-fourth}]
    Write $A = \{b^4:b\in B\}$, where $B\subseteq [N]$. Let $Y = \eta\log N$, where $\eta>0$ is a sufficiently small absolute constant, and let $\mathcal{P}$ be the set of primes
    \[p\equiv3\pmod 4,\qquad  p\leq Y.\]
    By the prime number theorem in arithmetic progressions, 
    \[\ell:=|\mathcal{P}|\gg \frac{\log N}{\log\log N}.\]
    Put $\Delta := \prod_{p\in \mathcal{P}}p$. Choosing $\eta$ sufficiently small, we can assume $\Delta\leq N^{1/10}$. 

For any $p\in\mathcal{P}$ and \((\xi_1,\xi_2,\xi_3,\xi_4)\in\mathbb F_p^4\), define the local weight function
$$
        m_p(\xi_1,\xi_2, \xi_3, \xi_4)
        :=\begin{cases}
            p^2,\ &\xi_1=\cdots=\xi_4=0\\
            p^2/(p+1),\ &\boldsymbol{\xi}\neq 0,\ F(\xi_1,\ldots,\xi_4) = 0\\
            0,\ &\text{otherwise}
        \end{cases}
$$

If $\xi_1\in \F_p\setminus \{0\}$, then by Lemma~\ref{lem:local4}, there are $p^2+p$ solutions to $F(\xi_1,\ldots,\xi_4) = 0$, and thus
\[
\sum_{\xi_2,\xi_3,\xi_4\in\mathbb F_p}
m_p(\xi_1,\xi_2, \xi_3, \xi_4)=(p^2+p)p^2/(p+1)=p^3.
\]
On the other hand, if $\xi_1=0$, then by Lemma~\ref{lem:local4}, there are $p^2$ solutions to $F(\xi_1,\ldots,\xi_4) = 0$, including the zero solution, thus
\[
\sum_{\xi_2,\xi_3,\xi_4\in\mathbb F_p}
m_p(\xi_1,\xi_2, \xi_3, \xi_4)=p^2+(p^2-1)p^2/(p+1)=p^2+p^2(p-1)=p^3.
\]
The same holds symmetrically in
the other coordinates.

Let $b_1,\ldots,b_4\in B$ and set $n = F(b_1,b_2,b_3,b_4)$. If $m_p(b_1,b_2,b_3,b_4)>0$, then $F(b_1,b_2,b_3,b_4)\equiv 0\pmod p$, so $p\mid n$. Moreover, $m_p(b_1,b_2,b_3,b_4) = p^2$ only when $b_i\equiv 0\pmod p$ for each $1\leq i\leq 4$, which implies $p^4\mid n$. Hence \[m_p(b_1,b_2,b_3,b_4)\leq \frac{p^2}{p+1}\mathbf{1}_{p\mid n}+p^2\mathbf{1}_{p^4\mid n}.\]

Now we can apply Theorem~\ref{thm:weight} to obtain
\[|B|^4\ll g^{1/4}N^4\exp\left(-\frac{1}{8}\sum_{p\in\mathcal{P}}\frac{1}{p+1}\right),\]
which implies 
\[|A|=|B|\ll g^{1/16}N\exp(-c\log\log Y) \ll g^{1/16}\frac{N}{(\log\log N)^c}\]
for some constant $c>0$.
\end{proof}
\subsection{Application III: Sidon-like sets for norm forms}

We first find some ``good" split primes in the following lemma. It is a standard application of the Chebotarev density
theorem.

\begin{lemma}[Good split primes]\label{lem:good-primes}
Let \(K\) be a number field of degree \(r\), and let \(\{\omega_1,\ldots,\omega_r\}\subseteq \cO_K\) be a fixed basis for $K$. There is a
set \(\mathscr P_K\) of rational primes of positive density and a finite set of
exceptional primes \(S_K\) such that, for every \(p\in \mathscr P_K\setminus S_K\),
the norm form factors modulo \(p\) as
\[
        F(x_1,\ldots,x_r)
        \equiv
        \prod_{j=1}^{r} L_{j,p}(x_1,\ldots,x_r)
        \pmod p,
\]
where \(L_{1,p},\ldots,L_{r,p}\) are pairwise non-proportional linear forms over
\(\mathbb F_p\), and every coefficient of every \(L_{j,p}\) is nonzero.
\end{lemma}

\begin{proof}
Let $\cO = \Z\omega_1\oplus\cdots\oplus\Z\omega_r$. First note that \(F\in \mathbb Z[x_1,\ldots,x_r]\). Indeed, if
\(\alpha=x_1\omega_1+\cdots+x_r\omega_r\), then multiplication by \(\alpha\) on
\(K\) is represented, with respect to some fixed integral basis
\(\sigma_1,\ldots,\sigma_r\), by a matrix $m_{\alpha}$ whose entries are integral linear forms
in \(x_1,\ldots,x_r\), and
\(F(x_1,\ldots,x_r)=N_{K/\mathbb Q}(\alpha)=\det(m_\alpha)\).
Moreover,
\[
        [x_i^r]F=N_{K/\mathbb Q}(\omega_i)\neq 0
        \qquad (1\le i\le r).
\]

Over \(\overline{\mathbb Q}\), the norm form splits as
\[
        F(x_1,\ldots,x_r)
        =
        \prod_{\sigma:K\hookrightarrow \overline{\mathbb Q}}
        \left(\sum_{i=1}^{r}\sigma(\omega_i)x_i\right).
\]
By the Chebotarev density theorem, the rational primes that split completely in
the Galois closure of \(K\) have positive density. For such primes, after excluding
the finitely many primes dividing the index of
\(\mathcal O\), we claim that the above factorization reduces modulo \(p\) to a factorization
into \(r\) pairwise non-proportional linear forms over \(\mathbb F_p\).

To see this, recall that for any rational prime $p$ that splits completely, we have
\begin{equation}\label{eq:split_prime}
\cO_K/(p)\cong \prod_{i=1}^r \cO_K/\mathfrak{P_i}\cong (\Z/p\Z)^r.
\end{equation}
By definition, $F(x_1,\ldots,x_r)$ is the norm of $\sum_i x_i\omega_i$, and modulo $p$ it equals the determinant of multiplication by $\sum_ix_i\omega_i$ in $\cO_K/(p)$. Let $\pi_i:\cO_K\rightarrow\cO_K/\mathfrak{P_i}$ be the quotient homomorphism. Then under the ring isomorphism in equation~\eqref{eq:split_prime}, multiplication by $\sum_ix_i\omega_i$ in $\cO_K/(p)$ corresponds to multiplication by $(\sum_ix_i\pi_j(\omega_i))_{1\leq j\leq r}$ in $(\Z/p\Z)^r$, which is a diagonal matrix. Hence
\[F(x_1,\ldots,x_r)\equiv \prod_{j=1}^r\left(\sum_{i=1}^r x_i\pi_j(\omega_i)\right)\pmod p.\]
Now if $p\nmid [\cO_K:\cO]$, then $\{\omega_1,\ldots,\omega_r\}$ forms a $\F_p$-basis for $\cO_K/(p)$, which implies that the vectors 
\[\{(\pi_j(\omega_i))_{1\leq j\leq r}: 1\leq i\leq r\}\]
are linearly independent over $\F_p$. Therefore the linear forms $\{\sum_i x_i\pi_j(\omega_i): 1\leq j\leq r\}$ are pairwise non-proportional. This finishes the proof of the claim.

It remains only to remove finitely many further primes. The condition that a
coefficient of one of the reduced factors vanish is given by the vanishing modulo \(p\) of finitely
many fixed nonzero algebraic integers. Equivalently, they exclude only finitely
many rational primes. Removing these primes from the positive-density Chebotarev
set gives the lemma.
\end{proof}

Next we present the proof of Theorem~\ref{thm:norm-almost-sidon}.

\begin{proof}[Proof of Theorem~\ref{thm:norm-almost-sidon}]
 Let \(\mathscr P_K\) be the positive-density set of good split primes from
Lemma~\ref{lem:good-primes}, with the finite exceptional set removed. Let $Y = \eta\log N$, where $\eta>0$ is a sufficiently small absolute constant, and let 
\[\mathcal{P} = \mathscr P_K\cap [Y,2Y],\quad \ell:= |\mathcal{P}|.\]
From Lemma~\ref{lem:good-primes}, there exists some absolute constant $\theta>0$ such that $\ell\geq \theta \log N/\log \log N$. Put 
\[\Delta:= \prod_{p\in \mathcal{P}}p.\]
Choosing $\eta$ sufficiently small, we can assume $\Delta\leq N^{1/10}$. For each $p\in\mathcal{P}$ and $(x_1,\ldots,x_r)\in \F_p^r$, define the local weight function:
$$
m_p(x_{1},\ldots, x_{r}):= \frac{p}{r}\#\{1\leq j\leq r: L_{j,p}(x_{1},\ldots, x_{r}) = 0\},
$$
where $L_{j,p}$'s are the linear factors defined in Lemma~\ref{lem:good-primes}.
After rearranging and scaling, each linear factor $L_{j,p}(x_1,\ldots,x_r) = 0$ becomes $x_r = \Phi_{j,p}(x_1,\ldots,x_{r-1})$ for some linear map $\Phi_{j,p}:\F_p^{r-1}\rightarrow\F_p$. Moreover, since the linear maps
\(\Phi_{j,p}\) are nonzero, each of them is surjective with all fibers of size \(p^{r-2}\). Hence for each $x_r\in \F_p$, we have
\[
\sum_{x_1,\ldots, x_{r-1}\in \F_p}m_p(x_{1},\ldots, x_{r}) = \frac{p}{r}\sum_{j=1}^rp^{r-2} = p^{r-1}.
\]
The same holds symmetrically in the other coordinates.

For any $(x_1,\ldots,x_r)\in A_1 \times A_2 \times \cdots \times A_r$, set $n = F(x_{1},\ldots, x_{r})$. If $m_p(x_{1},\ldots, x_{r})>0$, then $F(x_{1},\ldots, x_{r})\equiv 0\pmod p$, so $p\mid n$. Moreover, $m_p(x_{1},\ldots, x_{r})> p/r$ only when $L_{i,p}(x_{1},\ldots, x_{r})\equiv L_{j,p}(x_{1},\ldots, x_{r})\equiv 0\pmod p$ for some $i\neq j$, which implies $p^2\mid n$. Hence 
\[m_p(x_{1},\ldots, x_{r})\leq \frac{p}{r}\mathbf{1}_{p\mid n}+p\mathbf{1}_{p^2\mid n}.\] 
Now we can apply Theorem~\ref{thm:weight} to obtain
\[
\prod_{i=1}^r|A_i|\ll_F g^{1/r}N^r\exp\left(-\frac{1}{2r}\sum_{p\in\mathcal{P}}\left(\frac{r-1}{r}-p^{-1}\right)\right)\leq C_Fg^{1/r}N^r\exp\left(-c_F\frac{\log N}{\log\log N}\right)
\]
for some constant $C_F,c_F>0$, as required.
\end{proof}

\subsection{Other applications}
The following theorem can be viewed as a distance/difference version of Theorem~\ref{thm:weight}.
\begin{theorem}\label{thm:weighted_diff}
    Let $r\geq 2$. Suppose \(F(x_1,\ldots,x_r)\) is an integer-valued polynomial of degree at most \(r\) defined over rationals, and
\(A \subset [N]^r\). Suppose
\[
R_{A,F}(n)
:= \#\bigl\{(a,b)\in A\times A :
a\neq b,\ F(a-b)=n\bigr\}
\le g
\]
for every \(n\). Let $s\geq 2$ be a positive integer and $t\geq 1$ be a real number. Let $\mathcal{P}$ be a set of $\ell$ primes. Put $\Delta:=\prod_{p\in\mathcal P}p,$ and assume that \(\Delta\le N\) and $\Delta^s\leq N^r$. Suppose that for each \(p\in \mathcal{P}\), there is a local weight function \[m_p:\mathbb{F}_p^r\times\F_p^r \to \mathbb{R}_{\ge 0}\]
satisfying the following two conditions:

(1) For every fixed \(x\in \F_p^r\), we have
\[
\sum_{y} m_p(x,y) = \sum_{y}m_p(y,x) = p^r.
\]

(2) There exists \(\delta_p\in(0,1)\) such that, for every
\(\mathbf{a}\neq \mathbf{b}\in\mathbb Z^r\),
\[
        m_p(\overline{\mathbf{a}},\overline{\mathbf{b}})
        \le
        p(1-\delta_p)\mathbf 1_{p\mid F(\mathbf{a}-\mathbf{b})}
        +
        p^t\mathbf 1_{p^s\mid F(\mathbf{a}-\mathbf{b})},
\]
where \(\overline{\mathbf{b}} \in \F_p^r\) denotes the reduction of \(\mathbf{b}\) modulo \(p\).

Then
\[
|A|\ll_{F} g^{1/2}N^{r/2}
\exp\left(-\frac{1}{12}\sum_{p\in \mathcal{P}}(\delta_p-p^{t-s})\right)+2^{\ell/6}g^{1/2}N^{r/3} \Delta^{t/6}+2^{\ell/5}g^{2/5}N^{2r/5}\Delta^{t/5},
\]
where the implied constant only depends on $F$.
\end{theorem}

\begin{proof}
    The proof is nearly identical to the proof of Theorem~\ref{thm:weight}, except for the part of estimating the entropy defect.

    For $a,b\in A$, let
\[K(a,b) = \prod_{p\in\mathcal{P}}(1+m_p(a,b)).\]
Set 
\[S:=\sum_{a,b\in A}K(a,b).\]
We first prove a lower bound for $S$. Let
    \(Z_1,Z_2\) be independent random variables, both uniformly
distributed on \(A\). For each $1\leq i \leq 2$, let \(\overline Z_i\) denote the reduction of \(Z_i\)
modulo \(\Delta\), viewed as a random variable taking values in
\[
I=(\mathbb Z/\Delta\mathbb Z)^r
\simeq
\prod_{p\in\mathcal P}\mathbb F_p^r.
\]
By Proposition~\ref{prop:weightentropy},
\begin{equation}\label{eq:diff_lower_bd}
\mathcal S
=
|A|^2 \cdot \mathbb E K(\overline{Z_1},\overline{Z_2})
\ge
|A|^2 \cdot 2^\ell
\exp\left(
        -
        \bigl(2\log |I|-\Ent(\overline{Z_1})-\Ent(\overline{Z_2})\bigr)
        \right).
\end{equation}

Let $1\leq i \leq 2$. We now bound the entropy defect of \(\overline{Z_i}\). Note that $\mathbb P(\overline{Z_i} = x) = |A(x,\Delta)|/|A|$ for each $x\in I$. By Jensen's inequality,
\[\frac{1}{|A|^2}\sum_{x\in I}|A(x,\Delta)|^2 = \sum_{x\in I}\mathbb P(\overline{Z_i} = x)^2\geq \exp(-\Ent(\overline{Z_i})).\]
On the other hand, since $R_{A,F}(n)\leq g$ for every $n$, $A$ must be a $B_2^-[g]$ set. Thus for each vector $t\in [-(N/\Delta+1), (N/\Delta+1)]^r\cap\Z^r$, there are at most $g$ pairs $(a,b)\in A$ such that $a\neq b$ and $a-b = \Delta t$. It follows that
\[\sum_{x\in I}|A(x,\Delta)|^2\leq |A|+O\left(g\left(\frac{N}{\Delta}\right)^r\right)\ll g\left(\frac{N}{\Delta}\right)^r,\]
where we have used $|A|\ll g^{1/2}N^{r/2}$ as a consequence of the $B_2[g]$ property.
Combining these two bounds yields
\[\log \Delta^r-\Ent(\overline{Z_i})\leq \log\left(\frac{gN^r}{|A|^2}\right)+O(1).\]
Substituting this into inequality~\eqref{eq:diff_lower_bd}, we obtain
\[\mathcal{S}\geq |A|^2\cdot 2^\ell\exp\left(-2\log\frac{gN^r}{|A|^2}+O(1)\right)\gg \frac{2^\ell |A|^6}{g^2N^{2r}}.\]

We next prove an upper bound for $\mathcal{S}$. By the assumption, for each $(x_1,x_2)\in \F_p^r\times\F_p^r$, we have
\[
m_p(x_1,x_2)
\le
p(1-\delta_p)\mathbf{1}_{p \mid F(x_1-x_2)}+p^t \mathbf{1}_{p^s \mid F(x_1-x_2)}.
\]

Since $F$ is a polynomial of degree at most $r$, there is a constant $C\geq 1$ such that $|F(x_1-x_2)|\leq CN^r$ for all $x_1,x_2\in [N]^r$. Then, since $R_{A,F}(n)\leq g$ for all integer $n$, we have
\[\mathcal{S} = \sum_{\substack{a,b\in A\\a\neq b}}K(a,b)+\sum_{a\in A}K(a,a)\leq g\sum_{|n|\leq CN^r}\prod_{p\in\mathcal{P}}\left(1+p(1-\delta_p)\mathbf{1}_{p\mid n}+p^t\mathbf{1}_{p^s\mid n}\right)+|A|\prod_{p\in\mathcal{P}}(1+p+p^t).\]
Following the same estimates for multiplicative functions in the proof of Theorem~\ref{thm:weight}, we obtain
\[
\mathcal{S}\ll_F gN^r2^\ell\exp\left(-\frac{1}{2}\sum_{p\in\mathcal{P}}(\delta_p-p^{t-s})\right)+g3^\ell \Delta^t+|A|3^\ell\Delta^t.
\]

Comparing the lower and upper bounds on $\mathcal{S}$, we obtain 
\[
\frac{2^{\ell}(|A|)^6}{g^2N^{2r}}\ll_F gN^r2^\ell\exp\left(-\frac{1}{2}\sum_{p\in\mathcal{P}}(\delta_p-p^{t-s})\right)+(g+|A|)3^\ell \Delta^t,
\]
which implies 
\[
|A|\ll_F g^{1/2}N^{r/2}
\exp\left(-\frac{1}{12}\sum_{p\in \mathcal{P}}(\delta_p-p^{t-s})\right)+2^{\ell/6}g^{1/2}N^{r/3}\Delta^{t/6}+2^{\ell/5}g^{2/5}N^{2r/5}\Delta^{t/5},
\]
as required.
\end{proof}

\begin{theorem}
    There is an absolute constant $c>0$ such that if $A\subseteq [N]^3$ has no repeated $L^3$ distance (given by $\sum_{i=1}^3(x_i-y_i)^3$), then 
    \[|A|\ll N^{3/2}\exp\left(-c\frac{(\log N)^{1/2}}{\log\log N}\right).\]
\end{theorem}
\begin{proof}
    Define $Y,\mathcal{P},\ell,\Delta$ as in the proof of Theorem~\ref{thm:almost-sidon-cubes}.
For each $p\in\mathcal{P}$, define the local weight \(m_p:\mathbb F_p^3\times\F_p^3\to \mathbb R_{\ge 0}\) by
\[
m_p(x,y)=
\begin{cases}
w_p, & x_i\ne y_i\text{ for all }i\quad \text{and} \quad \sum_{i=1}^3(x_i-y_i)^3=0,\\
u_p, & \text{exactly one of }x_i = y_i,\quad \text{and} \quad \sum_{i=1}^3(x_i-y_i)^3=0,\\
0, & \text{otherwise},
\end{cases}
\]
where
\[
        u_p=\frac{p^2}{3(p-1)}, \qquad 
        w_p=\frac{p^2-6u_p}{p-a_p-8}.
\]
Next, we check the two conditions in Theorem~\ref{thm:weighted_diff}. For any $x\in\F_p^3$, when \(y_1=x_1\), there are \(3(p-1)\) nonzero solutions of $(y_2-x_2)^3+(y_3-x_3)^3=0$. While for $y_1\neq x_1$, there are \(6\) solutions with exactly one of \(y_i=x_i\) ($2\leq i\leq 3$) and \(p-a_p-8\) solutions with $y_2\neq x_2$ and $y_3\neq x_3$. Thus
\[
\sum_{y\in \F_p^3} m_p(x,y)=3(p-1)u_p+(p-1)(6u_p+(p-a_p-8)) = p^3.
\]
The same holds symmetrically in the other coordinate. 

Note that $u_p\leq w_p\leq p(1-\frac{\alpha}{\sqrt{p}})$ for some constant $\alpha>0$. Hence
\[m_p(x,y)\leq p\left(1-\frac{\alpha}{\sqrt{p}}\right)\mathbf{1}_{p\mid n}\]
for all $x,y\in\F_p^3$, where $n = \sum_{i=1}^3(x_i-y_i)^3$. It then follows from Theorem~\ref{thm:weighted_diff} that
\begin{align*}
|A|&\ll N^{3/2}\exp\left(-\frac{1}{12}\sum_{p\in\mathcal{P}}\frac{\alpha}{\sqrt{p}}\right)+N\Delta^{1/6}+N^{6/5}\Delta^{1/5}\\
&\ll N^{3/2}\exp\left(-c\frac{(\log N)^{1/2}}{\log\log N}\right)
\end{align*}
for some absolute constant $c>0$, where the last step follows from the prime number theorem and our choice of $Y,\ell$ and $\Delta$.
\end{proof}

\begin{theorem}
    There is an absolute constant $c>0$ such that if $A\subseteq [N]^4$ has no repeated $L^4$-distance, then 
    \[|A|\ll \frac{N^2}{(\log\log N)^c}.\]
\end{theorem}

\begin{proof}
    Define $Y,\mathcal{P},\ell,\Delta$ as in the proof of Theorem~\ref{thm:almost-sidon-fourth}. For each $p\in\mathcal{P}$ and \(x,y\in\mathbb F_p^4\), define the local weight function
$$
        m_p(x,y)
        :=\begin{cases}
            p^2,\ &x=y\\
            p^2/(p+1),\ &x\neq y,\ \sum_{i=1}^4(x_i-y_i)^4 = 0\\
            0,\ &\text{otherwise}
        \end{cases}
$$
By Lemma~\ref{lem:local4}, for any $p\in\mathcal{P}$, the equation $t_1^4+t_2^4+t_3^4+t_4^4=0$ has exactly $p^3+p^2-p$ solutions in $\F_p$, thus for each $x\in \F_p^4$, we have
\[\sum_{y\in\F_p^4}m_p(x,y) = (p^3+p^2-p-1)\frac{p^2}{p+1}+p^2 = p^4.\]
The same holds symmetrically in the other coordinate. 

Let $a,b\in A$ and set $n = \sum_{i=1}^4(a_i-b_i)^4$. If $m_p(a,b)>0$, then $p\mid n$. Moreover, $m_p(a,b) = p^2$ only when $a\equiv b\pmod p$, which implies $p^4\mid n$. Hence \[m_p(a,b)\leq \frac{p^2}{p+1}\mathbf{1}_{p\mid n}+p^2\mathbf{1}_{p^4\mid n}.\]

Now we can apply Theorem~\ref{thm:weighted_diff} with $F(x_1,x_2,x_3,x_4) = x_1^4+x_2^4+x_3^4+x_4^4$, $r=4$ and $g=1$ to obtain
\begin{align*}
|A|&\ll N^{2}
\exp\left(-\frac{1}{12}\sum_{p\in \mathcal{P}}(\frac{1}{p+1}-p^{-2})\right)+2^{\ell/6}N^{4/3}\Delta^{1/3}+2^{\ell/5}N^{8/5}\Delta^{2/5}\\
&\ll \frac{N^2}{(\log\log N)^c}
\end{align*}
for some constant $c>0$, where the last step follows from the prime number theorem and our choice of $Y,\ell$ and $\Delta$.
\end{proof}

\begin{remark}\label{rem:weightsubspace}
    The weighted sieve developed in this section is flexible enough to recover many of the
    difference-type applications proved earlier by the subspace sieve, although
    usually with a weaker value of the constant \(c\) in the exponential saving.
    For example, Corollary~\ref{cor:Sidonsquare} can be recovered from
    Theorem~\ref{thm:weight} by taking
    $F(x,y)=x^2-y^2$, and, for odd primes \(p\),
    \[
        m_p(x,y)
        :=
        \begin{cases}
            p,   & x=y=0,\\
            p/2, & xy\ne0 \text{ and } x=\pm y,\\
            0,   & \text{otherwise}.
        \end{cases}
    \]

    Similarly, Theorem~\ref{thm:no-repeated-distances} can be recovered from
    Theorem~\ref{thm:weighted_diff}.  For primes \(p\equiv1\pmod4\), we can
    take, for \(P=(x_1,y_1)\) and \(Q=(x_2,y_2)\),
    \[
        m_p(P,Q)
        :=
        \begin{cases}
            p,   & P=Q,\\
            p/2, & P\ne Q
                    \text{ and } (x_1-x_2)^2+(y_1-y_2)^2=0,\\
            0,   & \text{otherwise}.
        \end{cases}
    \]
    The same idea also recovers Theorem~\ref{thm:quadratic-form} and
    Theorem~\ref{thm:norm-forms}, by using the corresponding local
    weights coming from the splitting of the quadratic form or norm form modulo
    suitable primes.

    Thus the weighted sieve can recover, at least qualitatively, essentially all of
    algebraic upper-bound applications in the paper except for the fully
    general ill-distributed theorem, Theorem~\ref{thm:ill-distributed}. The direct subspace sieve remains useful because it is simpler, gives sharper
constants in the exponential saving for the bounded-difference applications,
and applies to arbitrary ill-distributed sets, whereas the weighted sieve
requires some additional local algebraic input beyond small residue image.
\end{remark}

\section{Concluding remarks}\label{sec:conclusion}
The arguments in this paper are driven by a common local-to-global mechanism:
an algebraic congruence splits into many compatible branches modulo many small
primes, and a global bounded-multiplicity hypothesis prevents all of these local
coincidences from occurring too often.  In the Sidon problem in
\(\{1^2,2^2,\ldots,N^2\}\), the local input is the elementary fact that squares
occupy about half of the residue classes modulo odd primes.  In the distance
problems, it is the factorization of
\[
        Q(x,y)\equiv0\pmod p
\]
into isotropic lines.  In the number-field setting, it is the splitting of rational
primes into prime ideals.

This viewpoint is closely related to inverse-sieve questions.  Hanson's theorem
shows that large sets occupying about half of the residue classes modulo every
prime must correlate additively with the squares \cite{HansonLargeSieve}.  Our
results use a complementary principle: when such local structure is combined with
Sidon-type or bounded-representation hypotheses, the forced modular coincidences
become too numerous unless the set is smaller than the naive large-sieve scale.

There are also some striking connections with  recent number-field constructions in discrete geometry and arithmetic combinatorics.  The counterexample to the Erd\H{o}s unit-distance conjecture \cite{OpenAIUnitDistance}, the subsequent counterexample to the real sum-product conjecture \cite{BloomSawinSchildkrautZhelezov}, and the latest counterexample to the Elekes--R\'onyai problem \cite{PohoataPaper} have all illustrated that high-degree number fields can produce finite real configurations with unexpectedly strong additive or geometric structure.

The closest parallel to the present paper is the split-prime construction for the Elekes--R\'onyai problem, recently introduced by the third author. There one considers polynomials of the form $f_Q(x,y)=Q(x+y)+(x-y)^2$, where \(Q\) is a squarefree product of rational primes splitting completely in a
suitable number field.  The split prime ideals impose many independent
congruence restrictions, while the squarefree modulus packages these restrictions
through the Chinese remainder theorem.  After the scaling identity $f_Q(Qx,Qy)=Q^2\bigl(x+y+(x-y)^2\bigr)$, one obtains small image sets for the fixed polynomial
$$f(x,y)=x+y+(x-y)^2,$$
which is neither of additive nor multiplicative, in the Elekes--R\'onyai sense.  Thus the
same philosophy appears in two dual forms: in this paper, local branching plus
bounded multiplicity gives upper bounds; in the Elekes--R\'onyai construction,
local branching is engineered to create many controlled coincidences. 

We expect that the idea from \cite{PohoataPaper} to combine the combinatorial sieve developed in this paper with the number field towers from \cite{GolodShafarevich} along with the companion split primes technology from \cite{HajirMaireRamakrishna2021} will lead to many further constructions for various other problems in combinatorics, number theory, or discrete geometry. For example, inspired by Theorem \ref{thm:no-isosceles}, let \(\operatorname{subset}'(n)\) denote the largest integer \(m\) such that every
\(n\)-point set in the plane contains an \(m\)-point subset with no isosceles
triangle.  Equivalently,
\[
        \operatorname{subset}'(n)
        =
        \min_{|P|=n}
        \max\{|A|:A\subseteq P,\ A\text{ contains no isosceles triangle}\}.
\]
Determining the asymptotics of $\operatorname{subset}'(n)$ is an old problem of Erd\H{o}s from \cite{Erdos80}. See also \cite{BMP}. Note that Theorem \ref{thm:no-isosceles} already implies the new bound $\operatorname{subset}'(n) \ll  n\exp\left(-c\frac{\log n}{\log\log n}\right)$. Nevertheless, combining the ideas from \cite{PohoataPaper} and the present paper, it is possible to further show that there is an absolute constant \(c>0\) such that
\begin{equation} \label{split1}
        \operatorname{subset}'(n)\ll n^{1-c}.
\end{equation}
This confirms a conjecture of Erd\H{o}s from 1980. See for example \cite{Erdos80}. 

Similarly, let \(\operatorname{subset}''(n)\) denote the largest integer \(m\)
such that every \(n\)-point planar set contains an \(m\)-point subset in which
no distance occurs more than once; that is,
\[
        \operatorname{subset}''(n)
        =
        \min_{|P|=n}
        \max\{|A|:A\subseteq P,\ A\text{ determines no repeated distance}\}.
\]
The problem of determining the asymptotics of $\operatorname{subset}''(n)$ is also an old problem of Erd\H{o}s, also appearing in the Erd\H{o}s Problems collection as Problem~\#1208
\cite{ErdosProblems1208}. In \cite{ClemenFuehrerRocheNewton}, it was recently shown that $\operatorname{subset}''(n) \gg n^{1/3}$. Note again that Theorem \ref{thm:no-repeated-distances} immediately implies the new upper bound $\operatorname{subset}''(n)\ll n^{1/2} \exp\left(-c \frac{\log n}{\log \log n}\right)$. Nevertheless, combining the ideas from \cite{PohoataPaper} and the present paper, it is also possible to show that
\begin{equation} \label{split2}
        \operatorname{subset}''(n)\ll n^{1/2-c}
\end{equation}
for some absolute constant \(c>0\) (and in fact, much more is true). These will be pursued in a separate paper. 

That being said, these examples also suggest a broader program. The combinatorial sieve developed here converts local algebraic branching into upper bounds under bounded-multiplicity hypotheses.  The recent number-field constructions show how to amplify such branching by using bounded-discriminant towers with many split prime ideals. Further combinations of these two perspectives should lead to new constructions for many different extremal problems. 

\section*{Acknowledgments}
The authors would like to thank Thomas Bloom, David Conlon, Will Sawin, and Yu-Chen Sun for helpful discussions. We would also like to acknowledge the usage of AI in preparation of this manuscript. All the mathematical ideas and final proofs are human generated.

\appendix

\section{Some auxiliary inequalities}
\begin{lemma}\label{lem:pointwise}
Let \(r\ge2\). For all \(x_1,\ldots,x_r\ge0\),
\[
        \prod_{i=1}^r x_i^r+\sum_{i=1}^r x_i^r
        \ge
        2r\sum_{i=1}^r x_i-(2r^2-r-1).
\]
\end{lemma}
\begin{proof}
    Using the AM-GM inequality, we get 
    \[\prod_{i=1}^rx_i^r+r-1\geq r\prod_{i=1}^rx_i.\]
    Hence it suffices to show that
    \begin{equation}\label{eq:amgm}
    \sum_{i=1}^rx_i^r+r\prod_{i=1}^rx_i-2r\sum_{i=1}^rx_i\geq -2r(r-1).
    \end{equation}
    Let $F(x_1,\ldots,x_r)$ denote the left-hand side of inequality~\eqref{eq:amgm}. We need to determine the global minimum of $F$ on the domain $\R_{\ge 0}^r$.

    \noindent\textbf{Case 1: Minimum on the boundary.}
    Suppose at least one of the variables is zero, say $x_r=0$. Then 
    \[F(x_1,\ldots,x_r) = \sum_{i=1}^{r-1}(x_i^r-2rx_i).\]
    The minimum value of $f(x) = x^r-2rx$ is attained at $x = 2^{1/(r-1)}$, thus 
    \[\sum_{i=1}^{r-1}(x_i^r-2rx_i)\geq -(r-1)^22^{r/(r-1)}.\]
    By the convexity of the function $2^x$, we have 
    \[2^{1/(r-1)}\leq 1+\frac{1}{r-1} = \frac{r}{r-1}\]
    since $r\geq 2$. Therefore
    \[2r(r-1)\geq (r-1)^22^{r/(r-1)},\]
    and after rearranging we obtain the desired inequality $F(x_1,\ldots,x_r)\geq -2r(r-1)$.

    \noindent\textbf{Case 2: Minimum in the interior.}
    In this case, we can solve by setting the gradient of $F$ to be zero. For each $1\leq i\leq r$,
    \[\frac{\partial F}{\partial x_i} = 0\implies rx_i^{r-1}+r\prod_{j\neq i}x_j-2r = 0\implies x_i^r-2x_i+P = 0,\]
    where $P = \prod_{j=1}^r x_j$. Summing over $i$ gives
    \begin{equation}\label{eq:crit_pt_sum}
    \sum_{i=1}^rx_i^r +rP -2\sum_{i=1}^rx_i=0.
    \end{equation}
    Fix any such critical point $(x_1,\ldots,x_r)$ , substituting equation~\eqref{eq:crit_pt_sum} into $F$ yields
    \[F(x_1,\ldots,x_r) = 2(1-r)\sum_{i=1}^rx_i.\]
    Hence it suffices to show that 
    \[\sum_{i=1}^rx_i\leq r.\]
    
    Since each $x_i$ satisfies $x_i^{r-1}+\prod_{j\neq i}x_j = 2$, sum over $i$ and then use the AM-GM inequality, we get $P\leq 1$. Now the polynomial $g(x) = x^r-2x+P$ has at most two positive roots by Descartes' rule of signs: one root $a\in (0,1]$ and one root $b\in [1,2)$. If $a = b = 1$, then $x_i=1$ for all $i$ and $\sum_ix_i = r$. If the two roots are distinct ($a<1<b$), let $m = \#\{1\leq i\leq r:x_i = b\}$. Then $P = a^{r-m}b^m$ and 
    \begin{equation}\label{eq:crit_pt}
    a^{r}-2a+a^{r-m}b^m = 0,\quad b^{r}-2b+a^{r-m}b^{m} = 0.
    \end{equation}
    If $m = 0$, then $a^r-2a+a^r = 0$, which forces $a=1$, a contradiction; similarly, if $m=r$, then $b^r-2b+b^r = 0$, which forces $b=1$, a contradiction.

    Next we assume that $1\leq m\leq r-1$. If $r=2$, then $m=1$ and the equations~\eqref{eq:crit_pt} become
    \[a^2-2a+ab = a(a+b-2) = 0,\quad b^2-2b+ab = b(a+b-2) = 0.\]
    Since $a,b>0$, we must have $a+b = 2$ and hence $\sum_ix_i = (r-m)a+mb = 2 = r$, as required.
    
    Now assume $r\geq 3$. The two equations~\eqref{eq:crit_pt} become
    \[a^{r-1}-2+a^{r-m-1}b^m = 0,\quad b^{r-1}-2+a^{r-m}b^{m-1} = 0.\]
    Taking the difference of these two equations, we obtain
    \begin{equation}\label{eq:diff_of_powers}
    a^{r-1}-b^{r-1} = (a-b)a^{r-m-1}b^{m-1}.
    \end{equation}
    Since $a\neq b$, $r\geq 3, 1\leq m\leq r-1$, and $a,b$ are both positive, we have 
    \[\frac{a^{r-1}-b^{r-1}}{a-b} = \sum_{j=0}^{r-2}a^jb^{r-2-j}>a^{r-m-1}b^{m-1},\]
    contradicting equation~\eqref{eq:diff_of_powers}.
\end{proof}

The following corollary can be viewed as a refinement of H\"older's inequality.

\begin{corollary}\label{cor:r}
Let \(r\ge2\), and let \(A_1,\ldots,A_r\) be nonnegative random variables on
the same probability space. Then
\[
        \mathbb E\prod_{i=1}^r A_i+
        \prod_{i=1}^r\mathbb E A_i
        \ge
        2\prod_{i=1}^r
        \left(\mathbb E A_i^{1/r}\right)^r .
\]
\end{corollary}

\begin{proof}
We may assume \(0<\mathbb E A_i^{1/r}<\infty\) for every \(i\). For each $1\leq i\leq r$, put $B_i=A_i^{1/r} /\mathbb E A_i^{1/r}$ so that \(\mathbb E B_i=1\). Then, it suffices to prove
\begin{equation}\label{eq:sum2}
        \mathbb E\prod_{i=1}^rB_i^r+
        \prod_{i=1}^r\mathbb E B_i^r\ge 2.
\end{equation}
By Lemma~\ref{lem:pointwise}, we have
\[
        \mathbb E\prod_{i=1}^rB_i^r+
        \sum_{i=1}^r\mathbb E B_i^r
        \ge r+1.
\]
On the other hand, Jensen's inequality gives \(\mathbb E B_i^r\ge1\), and hence
\[
        \prod_{i=1}^r\mathbb E B_i^r
        \ge
        1+\sum_{i=1}^r(\mathbb E B_i^r-1)
        =
        \sum_{i=1}^r\mathbb E B_i^r-r+1.
\]
Combining the last two displayed inequalities yields the desired inequality~\eqref{eq:sum2}.
\end{proof}

\end{document}